\documentclass{amsart}
\usepackage{amsmath}
\usepackage{amssymb}
\usepackage{amsthm,graphicx}
\usepackage[mathscr]{eucal}
\usepackage{tikz}\usetikzlibrary{positioning}
\theoremstyle{plain}
\newtheorem*{theorem*}{Theorem}
\newtheorem{theorem}{Theorem}[section]
\newtheorem{thm}[theorem]{Theorem}
\newtheorem{lemma}[theorem]{Lemma}
\newtheorem{lem}[theorem]{Lemma}

\newtheorem{example}[theorem]{Example}

\newtheorem{definition}[theorem]{Definition}

\newtheorem{problem}[theorem]{Problem}

\newcommand{\refe}[1]{\stackrel{\eqref{#1}}{=}}
\newcommand{\up}[1]{\textup{#1}}

\newcommand{\dom}{\operatorname{dom}}
\newcommand{\nf}{\operatorname{nf}}

\newcommand{\ran}{\operatorname{ran}}
\newcommand{\fix}{\operatorname{fix}}

\newcommand{\A}{\mathsf{a}}
\newcommand{\D}{\mathsf{d}}
\newcommand{\R}{\mathsf{r}}

\newcommand{\Dd}{\mathsf{D}}
\newcommand{\Rr}{\mathsf{R}}
\newcommand{\F}{\mathsf{Fun}}
\newcommand{\Fix}{\mathsf{fix}}

\newcommand{\compo}{\mathbin{;}}
\newcommand{\meet}{\mathbin{\cdot}}
\newcommand{\add}{\operatorname{add}}
\newcommand{\view}{\operatorname{view}}
\newcommand{\upset}{\operatorname{\uparrow}}

\newcommand {\set}[1] { \{ #1 \} } 

\def\SS{{\mathscr S}}

\def\restr #1{{\restriction_{#1}}}

\begin{document}

\title{The algebra of functions with antidomain and range}
\begin{abstract}
We give complete, finite quasiequational axiomatisations 
to algebras of unary partial functions 
under the operations of composition, domain, antidomain, range 
and intersection.  
This completes the extensive programme of classifying algebras 
of unary partial functions under combinations of these operations.
We look at the complexity of the equational theories and provide a 
nondeterministic polynomial upper bound.
Finally we look at the problem of finite representability and show that finite algebras
can be represented as a collection of unary functions over a finite base set
provided that intersection is not in the signature.
\end{abstract}
\author{Robin Hirsch}
\address{Department of Computer Science, University College London, UK}
\email{r.hirsch@ucl.ac.uk}
\author{Marcel Jackson}
\address{Department of Mathematics and Statistics, La Trobe University, Victoria, Australia}
\email{m.g.jackson@latrobe.edu.au}
\author{Szabolcs Mikul\'as}
\address{School of Computer Science and Information Systems, Birkbeck College, University of London, UK}
\email{szabolcs@dcs.bbk.ac.uk}

\thanks{The second author was supported by ARC Future Fellowship FT120100666 and ARC Discovery Project DP1094578.}
\maketitle

\section{Introduction} 

The abstract algebraic study of partial maps goes back at least to Menger~\cite{men} 
and the subsequent work of Schweizer and Sklar~\cite{SS1,SS2,SS3,SS4}.  
A large body of work has followed, some of it specifically building on the work of 
Schweizer and Sklar (such as Schein~\cite{sch:DR}) but numerous other contributions 
with independent motivation starting from semigroup theory 
(where there is a close relationship to ample and weakly ample semigroups; 
see Hollings~\cite{hollings} for a survey), category theory 
(where there is a very close connection with restriction categories \cite{coclac,cocman}) 
and constructions in computer science \cite{jacsto:ITE,jacsto:modal}.  
Moreover there is a close connection to the more heavily developed algebraic theory 
of binary relations; see Maddux~\cite{mad} or Hirsch and Hodkinson~\cite{hirhod} 
in general and articles such as Hollenberg~\cite{hollenberg} 
(which also delves into equational properties of partial maps) and 
Desharnais, M\"oller and Struth~\cite{DMS}, where the development is closer in nature 
to the theme of applications of the algebra of partial maps.  
Of course, often the motivation has been across several of these fronts at once, 
with much of the category-theoretic development focussed toward computer science motivation, 
and articles such as Jackson and Stokes~\cite{jacsto:01,jacsto:modal} and Manes~\cite{man}
attempting in part to provide new links between the various perspectives.  
We make no attempt at a full survey here.  Some further references are given below, but other discussion and history can be found in Schein's early (but already substantial) survey article \cite{sch70a} or \cite{sch79} and in articles such as \cite{cocman,hollings,jacsto:modal}.

In each of the above approaches, the fundamental operation of composition of partial maps 
is accompanied by additional operations capturing facets of what it means to be a partial map.
Operations modelling the domain of a partial map are particularly ubiquitous in the literature 
but other frequently occurring operations are those modelling range, intersection, fixset, 
domain-complement, as well as programming-specific constructions such as 
if-then-else and looping.  Perhaps the foremost goal in the development of an abstract approach 
to partial maps is the construction of a system of axioms that can be proved sound and complete
relative to fragments of the first order theory of systems of partial maps closed under 
the given operations.  In this context, the present article completes a long programme of
investigation by giving  sound and complete (and finite) axiomatisations for arguably the 
last two remaining natural families.
These families are also the richest, consisting of essentially all of the basic operations
combined.  (A caveat is that we do not claim completeness when looping is included: 
for sufficiently rich signatures, it is shown in Goldblatt and Jackson~\cite{goljac} 
that there is no recursive system of axioms that will be sound and complete 
for even just the equational properties of partial maps with looping.)  
We also show that the complexity of deciding the validity of equations is
co-{NP}, and is complete for this class for sufficiently rich signatures.

\section{Preliminaries: operations and representability}
By a (unary) \emph{function} on a set $X$, we mean a partial map $f\colon X\to X$. 
We use $\dom(f)$ and $\ran(f)$ to denote the domain and range of $f$ respectively, 
and often write  $(x,y)\in f$ for $f(x)=y$.  
Domain and range may be recorded as functions by way of unary operations: 
we define $\D(f)$, the \emph{domain} of $f$ (as a function), 
as the identity relation restricted to $\dom(f)$
\begin{align*}
\D(f)&:=\{(x,x)\mid\exists y\ (x,y)\in f\},
\end{align*}
and $\R(f)$, the \emph{range} of $f$, 
as the identity relation restricted to $\ran(f)$
\begin{align*}
\R(f)&:=\{(y,y)\mid\exists x\ (x,y)\in f\}.
\end{align*}
Given two functions $f$ and $g$ on $X$, 
their \emph{composition} $f\compo g$ is defined as
\begin{align*}
f\compo g&:=\{(x,y)\mid \exists z((x,z)\in f\  \&\  (z,y)\in g)\},
\end{align*}
that is, $(f\compo g)(x)=g(f(x))$,
while $f\meet g$ denotes their intersection
\begin{align*}
f\meet g&:=\{(x,y)\mid (x,y)\in f\  \&\  (x,y)\in g\}.
\end{align*}
Observe that if $f$ and $g$ are functions on $X$ then so are
$\D(f), \R(f), f\compo g$ and $f\meet g$.

Many of the above motivations give rise to other operations of importance. 
For example, while the domain operation $\D$ models the possibility modality of dynamic logic, 
necessity of dynamic logic is instead modelled by ``antidomain''. 
The \emph{antidomain} of $f\colon X \to X$ is the function
\begin{align*}
\A(f)&:=\{(x,x)\mid x\in X\ \&\ \forall y\ (x,y)\notin f\}.
\end{align*}
Note that $X$ occurs as a parameter in the definition of antidomain 
just like the top element occurs in the definition of complement in boolean set algebras.
Observe that $\D(f)$ can be defined using $\A$ as $\D(f)=\A(\A(f))$. 
Similarly, we can define the identity function $1'$ 
\begin{align*}
1'&:=\{(x,x)\mid x\in X\}
\end{align*}
and $0$ will denote the empty function.
Observe that we can define these constants as $0=\A(x)\compo x$ (any $x$) and $1'=\A(0)$.
A final operation we consider is the \emph{fixset} operation defined by 
\[
\Fix(f):=\{(x,x)\mid (x,x)\in f\},
\]
which can also be defined using $\cdot$ and $\D$ as $\Fix(f)=\D(f)\cdot f$, 
or using $\cdot$ and $\compo$ by $\Fix(f)=f\cdot (f\compo f)$.
We will consider only signatures containing $\compo$ and usually omit explicit mention 
of $\D$ if $\A$ is present, and $\Fix$ if  $\cdot$ is present.

Given a similarity type $\tau$ with operations from among $\set{\compo,\meet,\D,\R,\A,\Fix,0,1'}$,
a \emph{$\tau$-algebra of functions} is a family of functions on some set $X$,
the \emph{base} of the algebra, augmented with the operations in $\tau$ as defined above
(using $X$ as the parameter).
We will denote by $\F(\tau)$ the class of $\tau$-algebras of functions.
When a $\tau$-algebra $\mathscr{S}$ is isomorphic 
to an element of $\F(\tau)$, we say that $\mathscr{S}$
is (functionally) \emph{representable}.

Functional representability has been considered for many subsignatures of 
$\set{\compo,\meet,\A,\D,\R}$ as well as combinations involving fixset and set subtraction.  
When intersection is not present, the article~\cite{jacsto:amon} 
provides a table describing finite axiomatisability, 
and describing whether the class is a variety or a quasivariety.  
Although the class of functionally representable algebras for the signature 
$\set{\compo, \D}$ is a variety \cite{tro}, all other functional representation classes 
for signatures without interscection form proper quasivarieties. 
For the relatively weak signatures $\tau$ with 
$\{\compo\}\subsetneq \tau\subseteq\{\compo,\R,\Fix\}$, no finite axiomatisation is possible, 
but all remaining cases have known complete finite axiomatisations except for 
the strongest of the signatures, $\set{\compo,\A,\R}$.
In the present article we will give a complete finite quasiequational axiomatisation 
for the signature $\set{\compo,\A,\R}$.

Signatures involving $\meet$ are discussed in depth in the introductory sections
of~\cite{jacsto:modal}.  
Again, aside from the somewhat artificial case of $\set{\compo,\cdot,\R}$, 
all cases have known axiomatisations except for $\set{\compo,\meet,\A,\R}$.
In the present article we give a finite equational axiomatisation characterising 
functional representability for the signature $\set{\compo,\meet,\A,\R}$.  
This result and the  characterisation of representability in the signature 
$\set{\compo,\A,\R}$ 
solve problems posed in the final subsection of~\cite{jacsto:modal}.

As is explained in~\cite{jacsto:modal}, many  natural operations can be written in terms of 
the signature $\set{\compo,\meet,\A,\R}$.  
For example, set subtraction $f-g$ (as examined in Schein \cite{sch:sub})
 can be written as $\A(f\cdot g)\compo f$ and
fixset can be defined as $\Fix(x)=\D(x)\meet x$.
However, two additional operations that cannot be expressed by $\set{\compo, \meet, \A, \R}$ 
are the 
\emph{preferential union} operation~$\sqcup$ and the 
\emph{maximum iterate} operation~${}^\uparrow$.   
The preferential union of $f$ with $g$ is defined to be $f(x)$ if $f(x)$ is defined and $g(x)$ otherwise.  In other words it is $f(x)\cup \A(f)\compo g(x)$, a union which always returns a function on functional arguments.  Preferential union can express   \textit{if-then-else} statements:   
\texttt{if }$f$ \texttt{then }$g$\texttt{ else }$h=\D(f)\compo g\;\sqcup\; \A(f)\compo h$ and in fact $f\sqcup g$ coincides with \texttt{if }$\D(f)$ \texttt{then }$f$\texttt{ else }$g$ so is identical to the \emph{override}  operation of 
Berendsen et al.~\cite{BJSV}.  Similarly, the \emph{update} operation of \cite{BJSV}  has $f$ update $g$ given by \texttt{if }$\D(f)\compo \D(g)$ \texttt{then }$g$\texttt{ else }$f$, so that update is a derived term (given $\compo,\sqcup, \A$).
Other variants of union are discussed in~\cite[\S2.3.2]{jacsto:modal}.  
Our results for the signatures $\set{\compo,\A,\R}$ and $\set{\compo,\meet,\A,\R}$ 
are extended here to include preferential union, thus subsuming the axiomatisation 
in~\cite{BJSV} (the axioms for the weaker signatures considered in~\cite{jacsto:modal} 
also subsume those of~\cite{BJSV}).  

The semantics of the maximum iterate operation are given by 
$$
f^\uparrow=\bigsqcup_{n<\omega} f^n\compo \A(f).
$$  
The maximum iterate operation can express while statements: 
\texttt{while}$(d)p = (\D(d)\compo p)^{\uparrow}\compo \A(d)$.
See~\cite{jacsto:modal} for more on this. 
Our axiomatisability results extend to signatures including ${}^\uparrow$
if we restrict ourselves to finite algebras.

The signature $\set{\compo,\D,\R}$ is one of the most obvious signatures and not surprisingly 
was one of the earliest signatures to receive serious attention through a series of articles 
by Schweizer and Sklar~\cite{SS1,SS2,SS3,SS4}.  
No complete axiomatisation for representability in this signature was found until the work of
Schein~\cite{sch:DR}, who gave a complete, finite quasiequational axiomatisation and a proof 
that the class is not a variety.  Schein's elegant representation gives only infinite
representations for finite $\set{\compo,\D,\R}$-algebras.  
In~\cite{jacsto:DR} it was shown that this representation also preserves~$\cdot$ 
and~$\Fix$ when they satisfy appropriate (sound) axioms.   
Schein's method of representation is invoked at one stage of our own representation, 
and we use a modification to show that for many signatures 
$\set{\compo,\D,\R}\subseteq \tau\subseteq \set{\compo,\D,\A,\R,\Fix,\sqcup,{}^\uparrow,0,1'}$ 
there are finite representations for finite $\tau$-algebras.  
This solves the first question in~\cite[\S10]{jacsto:DR}.
We leave the case when $\meet$ is in $\tau$ open.

We also look at the computational complexity of the equational theories and
prove that the validity problem is in {\bf co-NP} in all cases
when 
$\tau\subseteq\set{\compo,\meet,\sqcup, \D,\A,\R,\Fix,0,1'}$, 
and it is {\bf co-NP-complete} provided $\set{\compo,\A}\subseteq\tau$.   
It follows from~\cite{goljac} that the validity problem is $\Pi_1^1$-hard when 
$\set{\compo,\A,\Fix,\uparrow}\subseteq\tau$.

\section{Axioms for composition, (anti)domain, range and intersection}
In this section we recall known axiomatizations for
$\F(\tau)$ where $\tau\subseteq\set{\compo,\meet,\D,\A,\R}$.
We will assume that $\compo\in\tau$ and that either
$\A$ or $\D$ is in $\tau$ as well.

We introduce some notations and conventions.
Given a (not necessarily representable) $\tau$-algebra 
$\mathscr{S}=(S,\tau)$
such that $\D$ is in $\tau$ or is defined as $\A\A$,
we define the set $D(S)$ of \emph{domain elements} as
\begin{equation*}
D(S):=\{s\in S\mid \D(s)=s\}.
\end{equation*}
Lower case Greek letters $\alpha,\beta,\delta,\gamma,\dots$ will denote domain elements.

Next we list some (quasi)equations that are known to be valid in representable algebras.      
Associativity of $\compo$ is assumed throughout.

Domain axioms:
\renewcommand{\theequation}{\Roman{equation}}
\begin{align}
&\D(x)\compo x=x,\label{eq:leftid}\\
&\D(x)\compo\D(y)=\D(y)\compo\D(x),\label{eq:Dcomm}\\
&\D(\D(x))=\D(x),\\
&\D(x)\compo \D(x\compo y)=\D(x\compo y),\label{eq:Dorder}\\
&x\compo \D(y)=\D(x\compo y)\compo x,\label{eq:Dtwisted}\\
\intertext{and some of their consequences:}
&\D(x)\compo \D(y)=\D(\D(x)\compo y),\label{eq:needed}\\
&\D(x\compo \D(y))=\D(x\compo y).\label{eq:leftcong}
\end{align}
Associative $\{\compo, \D\}$-algebras obeying
axioms~\eqref{eq:leftid}--\eqref{eq:Dtwisted} have been given many names
\cite{bat,coclac,FGG,jacsto:01,man,tro} but the name \emph{restriction semigroup} 
(sometimes, ``one-sided restriction semigroup'') has emerged as the modern standard.

Meet axioms:
\begin{align}
&x\compo (y\cdot z)=(x\compo y)\cdot (x\compo z),\label{eq:meettwisted}\\
&x\cdot y=\D(x\cdot y)\compo x.\label{eq:meetorder}
\end{align}

\begin{thm}\label{thm:Drep}
\begin{itemize}
\item[(1)] 
\up(Trokhimenko \cite{tro}\up; see also \cite{jacsto:01} or \cite{man} 
for Cayley representation.\up)
The class $\F(\compo,\D)$ is the class of restriction semigroups,
that is, it is finitely axiomatised by associativity 
and \eqref{eq:leftid}\up{--}\eqref{eq:Dtwisted}.
\item[(2)] 
\up(Dudek and Trokhimenko \cite{dudtro}\up, Jackson and Stokes \cite{jacsto:03}.\up)
The class $\F(\compo, \cdot,\D)$ is finitely axiomatised by associativity, 
\eqref{eq:leftid}\up{--}\eqref{eq:Dtwisted}, the semilattice axioms for~$\cdot$ 
and \eqref{eq:meettwisted} and \eqref{eq:meetorder}.
\end{itemize}
\end{thm}

Antidomain axioms
(in view of \eqref{eq:zero} below  we may write $0$ for $\A(x)\compo x$ (any $x$) 
and $1'$ for $\A(0)$):
\begin{align}
&\A(x)\compo x=\A(y)\compo y, \label{eq:zero}\\
&0\compo y=0, \label{eq:zero'}\\
&1'\compo y = y,\label{eq:identity}\\
&\A(x)\compo \A(y)=\A(y)\compo \A(x),\label{eq:Acomm}\\
& x\compo \A(y)=\A(x\compo y)\compo x, \label{eq:Atwisted}\\
&\alpha \compo x=\alpha \compo y\And  \A(\alpha)\compo x=
\A(\alpha)\compo y\ \Rightarrow\ x=y.\label{eq:leftunion}\\
\intertext{
An associative $\{\compo, \A\}$-algebra satisfying
\eqref{eq:zero}--\eqref{eq:leftunion} is called a 
\emph{modal restriction semigroup} \cite{jacsto:modal}.  
(Note that the law $x\compo 0=0$ assumed in \cite{jacsto:modal} 
follows from $x\compo 0=x\compo \A(x)\compo x=\A(x\compo x)\compo x\compo x=0$, 
and then  $x\compo 1'=x\compo \A(0)=\A(x\compo 0)\compo x=\A(0)\compo x=x$.)
Some consequences deduced in \cite{jacsto:modal} include:}
&\A(x)\compo \A(y)=\A(x)\compo \A(\A(x)\compo y),\label{eq:Aorder}\\
&\alpha \compo x=\alpha \compo y\And  \beta\compo x=\beta\compo y\ \Rightarrow\ 
(\alpha\vee\beta)\compo x=(\alpha\vee\beta)\compo y,\label{eq:realleftunion}\\
&\A(\alpha\compo x)\compo\A(\beta\compo x)=
\A((\alpha\vee\beta)\compo x),\label{eq:inunion}
\end{align}
where $\alpha\vee\beta:=\A(\A(\alpha)\compo\A(\beta))$
for domain elements $\alpha$ and $\beta$.

\begin{thm}\label{thm:Arep} \up(Jackson and Stokes \cite{jacsto:modal}.\up)
\begin{itemize}
\item[(1)] 
The class $\F(\compo,\A)$ is the class of modal restriction semigroups, that is, 
it is finitely axiomatised by associativity and
\eqref{eq:zero}\up{--}\eqref{eq:leftunion}.
\item[(2)] 
The class $\F(\compo, \cdot,\A)$ is finitely axiomatised by associativity, 
the semilattice axioms for~$\cdot$, \eqref{eq:meettwisted} and \eqref{eq:meetorder}
\up(where $\D:=\A\A$\up),
and the antidomain axioms \eqref{eq:zero}\up{--}\eqref{eq:leftunion}.
\end{itemize}
\end{thm}

\begin{proof}  
The proof is essentially covered in~\cite{jacsto:modal} but the signatures used there 
are slightly different, so some translation is needed.  
For each part, the axioms are easily verified in $\F(\compo, \A)$ and $\F(\compo, \cdot, \A)$,
respectively.  
Conversely, for the first part, let $(S, \compo, \A)$ be any associative algebra satisfying
\eqref{eq:zero}--\eqref{eq:leftunion}.  
Axioms \eqref{eq:zero}--\eqref{eq:identity} prove that  $0, 1'$ have the usual multiplicative
properties and so $(S, \compo, \A,0,1')$ satisfies the conditions 
of~\cite[Definition~3]{jacsto:modal}, hence by~\cite[Theorem~4]{jacsto:modal} 
it is isomorphic to a member of $\F(\compo, \A)$, as required.  
For the second part, let $(S, \compo)$ be associative, let $(S, \meet)$ be a  semilattice
and $(S, \compo, \meet, \A)$ satisfy \eqref{eq:meettwisted}--\eqref{eq:leftunion}.  
Define a binary operation~$\bowtie$ by $x\bowtie y=\D(x\meet y)\vee(\A(x)\meet\A(y))$.  
The intended meaning of $x\bowtie y$ is the identity function over the points 
where $x$ and $y$ do not disagree.  It is not hard to check  that $(S, \compo, \bowtie, 0)$
satisfies \cite[Definition~19]{jacsto:modal}, hence by~\cite[Theorem~20]{jacsto:modal} 
it is isomorphic to an algebra of functions where $f\bowtie g$ is the identity restricted 
to the points where $f$ and $g$ either agree or are both undefined.  
From this we can recover a representation of $(S, \compo, \cdot, \A)$ using 
$f\cdot g=(f\bowtie g)\compo f$ and $\A(f)=0\bowtie f$.
\end{proof}

It follows from \cite[Theorem 20]{jacsto:modal} that $\F(\compo, \cdot,\A)$ 
is a finitely based variety, 
and the implicational law \eqref{eq:leftunion} can be expressed equationally.

Range axioms:
\begin{align}
&x\compo \R(x)=x,\label{eq:rightid}\\
&\R(x)\compo \R(y)=\R(y)\compo \R(x),\label{eq:Rcomm}\\
&\R(\R(x))=\R(x),\\
&\R(x\compo y)\compo \R(y)=\R(x\compo y),\label{eq:Rorder}\\
&\R(\R(x)\compo y)=\R(x\compo y),\label{eq:rightcong}\\
&x\compo y=x\compo z\ \Rightarrow\  \R(x)\compo y=\R(x)\compo z.\label{eq:schein}
\end{align} 

Domain--range axioms:
\begin{align}
&\D(\R(x))=\R(x) \text{ and } \R(\D(x))=\D(x).\label{eq:DR}
\end{align}

Antirange axiom:
\begin{align}
&\R(\alpha\compo x)\vee\R(\beta\compo x)=\R((\alpha\vee\beta)\compo x).\label{eq:antirange}
\end{align}
The antirange axiom does not seem to have played a role in previous works; 
after an application of DeMorgan's Law it is an ``antirange'' dual to law~\eqref{eq:inunion}.

\begin{thm}\label{thm:DRrep}
\begin{itemize}
\item[(1)] 
\up(Schein \cite{sch:DR}.\up) 
The class $\F(\compo, \D, \R)$ is finitely axiomatised by associativity\up, 
\eqref{eq:leftid}\up{--}\eqref{eq:Dtwisted}
and \eqref{eq:rightid}\up{--}\eqref{eq:DR}.
\item[(2)] 
\up(Jackson and Stokes \cite{jacsto:DR}.\up) 
The class $\F(\compo, \cdot, \D, \R)$ is finitely axiomatised by associativity\up,
\eqref{eq:leftid}\up{--}\eqref{eq:Dtwisted}\up,
the semilattice axioms for $\cdot$ and \eqref{eq:meettwisted} and \eqref{eq:meetorder}, 
the range axioms~\eqref{eq:rightid}\up{--}\eqref{eq:rightcong}\up, 
and the domain--range axioms~\eqref{eq:DR}\up.
\end{itemize}
If $0$ is added to the signature then the laws $0\compo x=x\compo 0=\D(0)=0$ ensure 
that~$0$ can be correctly represented as well.
\end{thm}

Law \eqref{eq:schein} is omitted from Theorem~\ref{thm:DRrep}(2) 
because it is shown in \cite[Lemma 9.8]{jacsto:DR} to be redundant 
in the presence of the other axioms for the signature $\{\compo, \cdot, \D, \R\}$. 
In particular, $\F(\compo, \cdot, \D, \R)$ is a variety.

In the next section we will give finite axiomatisations of the classes 
$\F(\compo, \A, \R)$ and $\F(\compo, \cdot, \A, \R)$  
using suitable selections from the above axioms.

Finally we consider axioms for preferential union and maximal iterate.
We have the following two laws
\begin{align}
\D(x)\compo (x\sqcup y)=x,&\label{eq:sqcupdom}\\
\A(x)\compo (x\sqcup y)=\A(x)\compo y.\label{eq:uniondef}
\end{align}
As mentioned in the introduction, iteration is far more elusive.  
We list some axioms and Theorem~\ref{thm:cupmi} will state the extent to which they are 
known to be complete:
\begin{align}
\D(x)\compo x^\uparrow=x\compo x^\uparrow&\ \text {and }\ 
\A(x)\compo x^\uparrow = \A(x),\label{eq:mibasic}\\
\D(x)\compo y=\D(x)\compo y\compo \D(x)\ &\Rightarrow\ 
\D(x)\compo y^\uparrow=\D(x)\compo y^\uparrow\compo \D(x).\label{eq:dynamic}
\end{align}

\begin{thm}\label{thm:cupmi} 
\begin{itemize}
\item[(1)] \up(Jackson and Stokes \cite[\S3.3]{jacsto:modal}.\up) 
The class $\F(\compo, \A, \sqcup)$ is characterised by the laws characterising the class
$\F(\compo,\A)$ together with laws \eqref{eq:sqcupdom} and \eqref{eq:uniondef}. 
Moreover, when laws \eqref{eq:sqcupdom} and \eqref{eq:uniondef} hold, 
every representation in the reduct signature $\F(\compo,\A)$ correctly represents~$\sqcup$.
\item[(2)] \up(Jackson and Stokes \cite{jacsto:modal}.\up)  
The finite
members of $\F(\compo, \A, \sqcup,{}^\uparrow)$ are characterised by the set of laws
characterising the signature $\{\compo,\A,\sqcup\}$ together with laws \eqref{eq:mibasic} 
and \eqref{eq:dynamic}.  
Moreover, when laws \eqref{eq:mibasic} and \eqref{eq:dynamic} hold, 
every representation in the reduct signature $\{\compo,\A,\sqcup\}$ correctly 
represents~${}^\uparrow$ for finite algebras.
\end{itemize}
\end{thm}

More is true: in part~(1),  the implication \eqref{eq:realleftunion} used in the 
characterisation of $\F(\compo, \A)$ becomes redundant \cite[Proposition 15]{jacsto:modal}, 
while in part~(2), the implication \eqref{eq:dynamic} can be replaced by the less 
intuitive equational law 
$\D(x)\compo \A(y\compo \A(x))\compo  (y\compo  \A(\D(x)\compo y\compo \A(x))^\uparrow \compo \A(x)=0$ 
\cite[Proposition 29]{jacsto:modal}.  
Thus in both cases one can obtain purely equational axioms.  

\renewcommand{\theequation}{\arabic{equation}}
\section{Characterisation of semigroups of functions with range and antidomain}\label{sec:char}
We are ready to state our main finite axiomatisability results.

\begin{thm}\label{thm:main}
The classes $\F(\compo, \A, \R)$ and $\F(\compo, \cdot, \A, \R)$ 
are finitely axiomatizable.
\begin{enumerate}
\item 
An algebra $(S,\compo,\A,\R)$ is representable 
as an algebra of functions if and only if
it is a modal restriction semigroup 
satisfying  laws
\eqref{eq:rightid}--
\eqref{eq:antirange}.  
The class of representable algebras is a proper quasivariety.
\item   
An algebra $(S, \compo, \cdot, \A, \R)$ is representable 
as an algebra of functions if and only if both $(S, \compo, \cdot,\A)$  and 
$( S,\compo,\cdot,\D,\R)$ are representable as algebras of functions 
and the law \eqref{eq:antirange} holds.
The class of representable algebras is a variety, it is axiomatised by: 
associativity, \eqref{eq:leftid}\up{--}\eqref{eq:Dtwisted}\up,
semilattice axioms for~$\cdot$, 
\eqref{eq:meettwisted}\up{--}
\eqref{eq:leftunion} and 
\eqref{eq:rightid}\up{--}
\eqref{eq:antirange}.
\end{enumerate}
With the additional axioms \eqref{eq:sqcupdom} and \eqref{eq:uniondef}\up, 
both items~(1) and~(2) extend to include a complete axiomatisation for 
functionally representable algebras with~$\sqcup$ in the signature.  
With axioms \eqref{eq:mibasic} and \eqref{eq:dynamic} adjoined 
these characterisations extend to include maximal iterate in the case of finite algebras.
\end{thm}

The proof of Theorem~\ref{thm:main} takes the following form.    
First we note that \cite[Proposition~11]{jacsto:modal} states that $\F(\compo, \A, \R)$ 
is a proper quasivariety, and it follows from \cite[Theorem~20]{jacsto:modal} that 
$\F(\compo, \cdot, \A, \R)$ is a variety.
For each case, the stated axioms are routinely seen to be valid.
Indeed, with the possible exception of \eqref{eq:antirange}, 
they consist of combining known axiomatisations for representability in various 
fragments of the signatures.
Aside from \eqref{eq:antirange}, the set of axioms stated in the first part 
is the union of an axiomatisation of  $\F(\compo,\D,\R)$ and an axiomatisation of 
$\F(\compo,\A)$.
Similarly the set of axioms implicit in the second part 
is the union of an axiomatisation of $\F(\compo, \cdot,\A)$ and 
an axiomatisation of $\F(\compo, \cdot, \D,\R)$.     
Hence in each case the stated axioms are valid over the class.  
The final two sentences of Theorem~\ref{thm:main} are a direct corollary of the ``moreover'' 
statements in Theorem~\ref{thm:cupmi}.  The comments after Theorem~\ref{thm:cupmi} 
show that in these settings one may obtain axiomatisations that are purely equational.

We demonstrate a faithful representation 
for any algebra satisfying the stated axioms. 
The construction is essentially the same in both instances.  
For a given algebra $\mathscr{S}=( S,\compo,\A,\R)$ satisfying the stated axioms, 
we construct a new algebra $\mathscr{S}^\flat$, 
and show that it satisfies the axioms for representability 
as an algebra of functions with composition, domain and range (but not with antidomain).  
If $\mathscr{S}$ also carries a semilattice satisfying the stated axioms, 
then $\mathscr{S}^\flat$ will also carry a semilattice
and satisfy the axioms for representability as an algebra of functions 
with composition, domain, range and intersection.  
We then show how to represent elements of $( S,\compo,\A,\R)$ 
as a union of elements in the representation of 
$\mathscr{S}^\flat$ (so that antidomain is preserved as well). 

Any restriction semigroup $(S,\compo,\D)$ carries a natural order relation defined by 
$x\leq y$ iff $x=\D(x)\compo y$, or equivalently, iff $\exists z\ x=\D(z)\compo y$.  
This natural order coincides with the relation 
\begin{equation*}
\set{(x\compo\alpha\compo y, \; x\compo y)\mid \alpha=\D(z), \; x, y, z\in S},
\end{equation*}
is stable under left and right multiplication and is represented as $\subseteq$ in 
$\{\compo,\D\}$-representations  
(these properties are routine syntactic consequences of axioms
\eqref{eq:leftid}--\eqref{eq:needed}, and can alternatively be deduced from 
Theorem~\ref{thm:Drep}(1); see~\cite{jacsto:01} or~\cite{man} for example).
Any model of the axioms in Theorem~\ref{thm:main} is a restriction semigroup 
with respect to the derived operation $\D$, and the corresponding natural order 
is used to construct~$S^\flat$.

In the case of a finite algebra $( S,\compo,\A,\R)$, 
the idea is quite straightforward: 
the set~$S^\flat$ can be interpreted as the union of~$\{0\}$ 
along with the set of atoms in the natural order~$\leq$.  
In general, we use an ultrafilter construction.

Throughout, we consider a fixed algebra 
$\mathscr{S}=(S,\compo,\A,\R)$ 
(or $( S,\compo,\cdot,\A,\R)$) satisfying the appropriate axioms 
from Theorem~\ref{thm:main}.  

\begin{lemma}\label{lem:BA}
\begin{enumerate}
\item
$( D(S), \compo, \A,0,1')$ forms a boolean algebra where $0, 1'$ 
are the bottom and top elements, respectively,  $\compo$ is boolean meet 
and $\A$ is boolean complement.
\item 
$\leq$ is an order relation \up(reflexive, transitive, antisymmetric\up). \label{part:mon}
For $s, t, u\in S$, if $s\leq t$ then 
$s\compo u\leq t\compo u,\; u\compo s\leq u\compo t$ and $\D(s)\leq \D(t)$.
\item\label{part:=}  
For $s, t\in S$, if $s\leq t$ and $\D(s)=\D(t)$ then $s=t$.
\item\label{part:vee}  
For $s, t\in S$ and $\alpha, \beta\in D(S)$, if $s\geq\alpha\compo t$ and 
$s\geq\beta\compo t$ then $s\geq(\alpha\vee\beta)\compo t$.
\item\label{part:vee2}  
For $s\in S$ and $\alpha, \beta\in D(S)$,
$\alpha\compo s\leq(\alpha\vee\beta)\compo s$ and
$\beta\compo s\leq(\alpha\vee\beta)\compo s$.
\end{enumerate}
\end{lemma}

\begin{proof}
The first part may be proved directly from the antidomain axioms, 
or alternatively it follows from the first part of Theorem~\ref{thm:Arep}.  

Properties (2)--(3) are discussed at the introduction of the natural order 
and are basic properties of restriction semigroups (see \cite{jacsto:01} for example), 
and also follow  from Theorem~\ref{thm:Drep}.

For the fourth part, suppose $s\geq \alpha\compo t$ and $s\geq \beta\compo t$.  
Then $\alpha\compo \D(t)\compo s\geq\alpha\compo t$ and 
if we apply the domain operation to both sides we get 
$\alpha\compo \D(t)\compo \D(s)\geq\alpha\compo \D(t)\geq\alpha\compo \D(t)\compo \D(s)$ 
and the two terms are equal.  
Hence, by the third part, $\alpha\compo \D(t)\compo s=\alpha\compo t$, 
and similarly $\beta\compo \D(t)\compo s=\beta\compo t$.  
By~\eqref{eq:realleftunion}, 
$(\alpha\vee\beta)\compo \D(t)\compo s=(\alpha\vee\beta)\compo t$, 
so $s\geq(\alpha\vee\beta)\compo \D(t)\compo s=(\alpha\vee\beta)\compo t$, as required.

The fifth part follows from part two, since $\alpha, \beta \leq\alpha\vee \beta$. 
\end{proof}

Since $D(S)$ forms a boolean algebra, there are  ultrafilters in $D(S)$.
We will denote these by capital Greek letters.  
Also note that the natural order~$\leq$, when restricted to elements in $D(S)$, 
agrees with the boolean ordering: $\alpha\leq \beta$ if and only if $\alpha\compo \beta=\alpha$  
(this follows from Theorem~\ref{thm:Drep} for example).
For a subset $X\subseteq S$ of elements, we define
\begin{equation*}
\upset X:=\{s\in S\mid (\exists x\in X) x\leq s\},
\end{equation*}
the \emph{upset} of $X$ in $S$.  
At times we will consider the restriction to domain elements of the upset of $X$, 
namely $\upset X\cap D(S)$.  An upset $X=\upset X$ is \emph{down directed} 
if, for every $x,y\in X$, there exists $z\in X$ with $z\leq x$ and $z\leq y$.
In the following, we extend the operations on $S$ pointwise to subsets of $S$, 
and treat elements of $S$ as singletons when convenient.  
So for example, for any $s\in S$ and $X\subseteq S$, we let 
$s\compo X:=\{s\compo x\mid x\in X\}$, 
while $\D(X)\compo X=\{\D(x)\compo y\mid x, y\in X\}$.

\begin{definition}
Let $\Delta$ be an ultrafilter of $D(S)$ and $s\in S$.  
If $0\notin \Delta\compo s$ then we say that the upset 
$\upset(\Delta\compo s)$ of $S$ is an \emph{ultrasubset} $S$.  
\end{definition}

Note that the upset (in $S$) of any ultrafilter $\Delta$ of $D(S)$ 
is an ultrasubset, because $\upset(\Delta\compo 1')=\upset\Delta$.  

\begin{lem}\label{lem:allultra}
\begin{enumerate}
\item[(i)] 
If $\Delta$ is an ultrafilter of $D(S)$ and $a\in S$ has $0\notin \Delta\compo a$
then $\Delta=\upset\D(\Delta\compo a)\cap\D(S)$, that is, $\Delta$ is the upset of 
$\D(\Delta\compo a)$ in $D(S)$ and so it is the unique ultrafilter of $D(S)$ 
extending $\D(\Delta\compo a)$.
\item[(ii)]
The following are equivalent for a subset $U\subseteq S$:
\begin{itemize}
\item 
$U$ is an ultrasubset of $S$,
\item 
$U$ is a down-directed filter of $S$ with respect to $\leq$
and $\upset\D(U)\cap D(S)$ is an ultrafilter of $D(S)$,  
\item 
$U$ is a maximal proper down-directed filter of $S$ with respect to $\leq$.
\end{itemize}
Moreover, for any ultrasubset $U$ and $a\in U$, 
we have $U=\upset(\D(U)\compo a)$.
\item[(iii)] 
If $\upset(\Delta_1\compo s_1)$ and $\upset(\Delta_2\compo s_2)$ are ultrasubsets of $S$ with 
$\upset(\Delta_1\compo s_1)\subseteq \upset(\Delta_2\compo s_2)$
then $\Delta_1=\Delta_2$ and 
$\upset(\Delta_1\compo s_1)=\upset(\Delta_2\compo s_2)$.
\item[(iv)] 
If $\upset(\Delta\compo s)\cap \upset(\Delta\compo t)\neq \varnothing$ 
then $\upset(\Delta\compo s)= \upset(\Delta\compo t)$.
\item[(v)] 
If $0\notin\Delta\compo s$ then the upset of $\R(\Delta\compo s)$ in $D(S)$ 
is an ultrafilter of $D(S)$.
\end{enumerate}
\end{lem}

\begin{proof}
(i)  
This follows because for every $\delta\in \Delta$, 
the element $\D(\delta\compo a)$ is in $\D(\upset(\Delta\compo a))$ 
and $\D(\delta\compo a)\leq \delta$. 

(ii)
First assume that $U=\upset(\Delta\compo a)$ is an ultrasubset of $D(S)$.  
As $\Delta$ (an ultrafilter of $D(S)$) is down directed, 
it follows that so is $U$, because $(\gamma\compo\delta)\compo a$ is a lower bound of both
$\gamma\compo a$ and $\delta\compo a$.  
Also,  $\upset\D(U)\cap D(S)=\Delta$ from part (i). 

Next, assume that $U$ is a down-directed filter with $\D(U)$ 
generating an ultrafilter of $D(S)$.  
We show that $U$ is a maximal down-directed filter.  
Assume $V$ is a down-directed filter with $U\subseteq V$ but $0\notin V$, we show that $U=V$.
Consider any $a\in V$. 
Then for any $b\in U$, as $U\subseteq V$, 
we have a lower bound $c\in V$ for $\{a,b\}$.  
Now $\D(c)\in \upset\D(U)$, because otherwise $\A(c)\in \upset\D(U)$ 
(as $\upset\D(U)\cap D(S)$ is an ultrafilter of $D(S)$), 
which would give an element $d\in U$ with $\D(d)\leq \A(c)$.  
But then $\D(d)\compo \D(c)=0$, so there could be no lower bound for $d$ and $c$ in $V$.  
So $\D(c)\in \upset\D(U)$.
Thus there is $u\in U$ with $\D(u)\leq\D(c)$.
Let $u'\in U$ be a lower bound of $u$ and $b$.
Then $\D(u')\compo b=u'\in U$, whence
$\D(u')\compo b\leq\D(u)\compo b\leq \D(c)\compo b\in U$.
But $c=\D(c)\compo b$ and $U$ is upward closed, showing that $a\in U$.  Thus $U=V$.

Now assume that $U$ is a maximal down-directed filter with respect to $\leq$.
Consider any $a\in U$. We claim that $\upset(\D(U)\compo a)=U$.  
For the inclusion $\upset(\D(U)\compo a)\subseteq U$, 
consider $\D(x)\compo a$ for some $x\in U$.  
Let $z\in U$ have $z\leq x$ and $z\leq a$, and note that $\D(x)\compo z=z$.  
Then $\D(x)\compo a\geq \D(x)\compo z=z$, so $\D(x)\compo a\in U$ as claimed.
For the reverse inclusion, consider $x\in U$.  
Then there is $z\in U$ with $z\leq x$ and $z\leq a$.  
So $\D(z)\in \D(U)$ and $\D(z)\compo a=z=\D(z)\compo x$.  
Then $x\geq \D(z)\compo x=\D(z)\compo a$, an element of $\D(U)\compo a$.  
So $x\in \upset(\D(U)\compo a)$ giving $\upset(\D(U)\compo a)= U$.  
Note that this also shows that $U$ can be written as 
$\upset(\D(U)\compo a)$ for any element $a\in U$, the final claim of part~(ii).

(iii) 
Part (i) gives $\Delta_1=\Delta_2$ and the final statement of part (ii) gives
$\upset(\Delta_1\compo s_1)=\upset(\Delta_2\compo s_2)$. 

(iv) 
This also follows from the final statement of (ii).
Once the two ultrasubsets have a common element $c$ 
then both can be written as $\upset(\Delta\compo c)$.

(v) 
To show that $\upset\R(\Delta\compo s)$ is a filter, 
it suffices to show that if $\delta_1,\delta_2\in \Delta$ 
then $\R(\delta_1\compo s)\compo\R(\delta_2\compo s)$ is in $\upset\R(\Delta\compo s)$.  
This follows because $\delta:=\delta_1\compo\delta_2$ is in $\Delta$ and 
$\R(\delta\compo s)$ is a lower bound of both $\R(\delta_1\compo s)$ and 
$\R(\delta_2\compo s)$ in $\R(\Delta\compo s)$.  

Now we 
show that it is an ultrafilter.
Assume that $\alpha\notin \upset\R(\Delta\compo s)$,
so $s\compo \alpha\notin \Delta\compo s$.  
Hence $\D(s\compo \alpha)\notin \upset \D(\Delta\compo s)=\upset \Delta$.  
So $\A(s\compo \alpha)\in\Delta$ showing that 
$s\compo \A(\alpha)\refe{eq:Atwisted}\A(s\compo \alpha)\compo s\in \Delta\compo s$.  
Then $\A(\alpha)\geq \R(\A(s\compo \alpha)\compo s)\in \R(\Delta\compo s)$, 
giving $\A(\alpha)\in \upset\R(\Delta\compo s)$.  
So the upset of $\R(\Delta\compo s)$ in $D(S)$ is an ultrafilter.
\end{proof}

\begin{lem}\label{lem:ultrasubsetfacts}
\begin{enumerate}
\item[(i)] 
$\upset\left(\upset(\Delta_1\compo s_1)\compo\upset(\Delta_2\compo s_2)\right)
=\upset(\Delta_1\compo s_1\compo \Delta_2\compo s_2)$.
\item[(ii)] 
If $0\notin \upset(\Delta_1\compo s_1\compo \Delta_2\compo s_2)$ then 
$\upset(\Delta_1\compo s_1\compo \Delta_2\compo s_2)=\upset(\Delta_1\compo (s_1\compo s_2))$.
\item[(iii)] 
$\upset\left(\upset(\Delta_1\compo s_1)\meet\upset(\Delta_2\compo s_2)\right)
=\upset\left(\left(\Delta_1\compo s_1\right)\meet 
\left(\Delta_2\compo s_2\right)\right)$.  
\item[(iv)]
If $0\notin\left(\Delta_1\compo s_1\right)\meet \left(\Delta_2\compo s_2\right)$ 
then $\Delta_1=\Delta_2$ and $\upset(\Delta_1\compo s_1)=\upset(\Delta_2\compo s_2)$ 
and 
$\upset\left( (\Delta_1\compo s_1\right)\meet \left(\Delta_2\compo s_2)\right)
=\upset \left(\Delta_1\compo (s_1\meet  s_2)\right)$.
\end{enumerate}
\end{lem}

\begin{proof}
Part~(i) is trivial, as is part~(iii).
For part~(ii) note that 
$\upset(\Delta_1\compo s_1\compo \Delta_2\compo s_2)\supseteq 
\upset(\Delta_1\compo s_1\compo s_2)$ 
and that $\upset(\Delta_1\compo s_1\compo s_2)$ is maximal by Lemma~\ref{lem:allultra} part~(ii). 
For part~(iv), observe that 
$\upset\left(\left(\Delta_1\compo s_1\right)\meet\left(\Delta_2\compo s_2\right)\right)
=\upset\left(\left(\Delta_1\meet\Delta_2\right)\compo \left(s_1\meet s_2\right)\right)$ 
is certainly a down-directed filter, moreover, one that contains both 
$\Delta_1\compo s_1$ and $\Delta_2\compo s_2$.  
The statement now follows from Lemma~\ref{lem:allultra} parts~(ii) and~(iii).
\end{proof}

\begin{definition}\label{defn:Sflat}
Let $S^\flat$ consist of $S$ along with the set of ultrasubsets of $S$.  
We define operations $\Dd$, $\Rr$, $\mathop{*}$ and $\wedge$ on $S^\flat$ 
corresponding to $\D$, $\R$, $\compo$ and $\meet$ \up(if present\up).
\begin{enumerate}
\item 
$\Dd(\upset(\Delta\compo s)):=\upset\D(\Delta\compo s)=
\upset(\Delta\compo \D(s))$ \up($=\upset\Delta$ when $\D(s)\in\Delta$, 
by Lemma~\ref{lem:allultra}(i)\up),
\item 
$\Rr(\upset(\Delta\compo s)):=\upset\R(\Delta\compo s)$,
\item 
$\upset(\Delta_1\compo s_1)\mathop{*}\upset(\Delta_2\compo s_2)
:=\upset\left((\Delta_1\compo s_1)\compo(\Delta_2\compo s_2)\right)$,
\item 
\up(if $\cdot$ is present\up) 
$\upset(\Delta_1\compo s_1)\wedge\upset(\Delta_2\compo s_2)
:=\upset\left((\Delta_1\compo s_1)\cdot(\Delta_2\compo s_2)\right)$,
\end{enumerate}
where $\Delta, \Delta_1, \Delta_2$ are ultrafilters of $D(S)$ and $s, s_1, s_2\in S$.
\end{definition}

Note that, while $S$ is not an ultrasubset by definition (as $0\in S$), 
it is covered by Definition~\ref{defn:Sflat} because $S=\upset(\Delta\compo 0)$.  
It is easy to see from Definition~\ref{defn:Sflat} that 
$S$ acts as an absorbing zero element with respect to both $*$ and $\wedge$, 
and that it is fixed by both $\Rr$ and $\Dd$ 
(simply because $0$ is fixed by $\D$ and $\R$ and is contained in $S$).  
Observe that Lemma~\ref{lem:ultrasubsetfacts} shows that in parts~(3) and~(4) 
of this definition either the right hand side  will  be $S$ 
or it can be written in the form $\upset(\Gamma\compo t)$.
Hence $S^\flat$ is indeed closed under the above operations.

\begin{lem}\label{lem:Sflat}
If $\mathscr{S}=(S,\compo,\A,\R)$ 
\up(or $\mathscr{S}=(S,\compo,\meet,\A,\R)$\up)
satisfies the axioms of Theorem~\ref{thm:main} 
then $\mathscr{S}^\flat=( S^\flat,\mathop{*},\Dd,\Rr)$ 
\up(or $\mathscr{S}^\flat=( S^\flat,\mathop{*},\wedge,\Dd,\Rr)$,
respectively\up) is representable. 
Furthermore, we can assume that the constant element $S\in S^\flat$ is represented as 
the empty function.
\end{lem}

\begin{proof}
It suffices to show that $\mathscr{S}^\flat$ satisfies 
the axioms for representability so that we can apply the
corresponding representation theorems from the previous section.
Checking the axioms is almost immediate because the operations are 
determined elementwise.  
For example, we verify \eqref{eq:Dtwisted}.  Consider ultrasubsets $U,V$.  
Now $x\in U\mathop{*}\Dd(V)$ if and only if $x\geq a\compo \D(b)$ 
for some $a\in U$ and $b\in V$.  
But $a\compo \D(b)=\D(a\compo b)\compo a$ which is an element of 
$\D(U\compo V)\compo U$, a subset of $\Dd(U*V)*U$.  
For the reverse inclusion, let $x\in\Dd(U*V)*U$.  
There are $a, c\in U$ and $b\in V$ such that $x\geq \D(a\compo b)\compo c$.  
Let $e\in U$ be a lower bound of $a, c$. 
Then $x\geq\D(e\compo b)\compo e=e\compo \D(b)\in U*\Dd(V)$, as required.
\end{proof}

\begin{proof}[Proof of Theorem \ref{thm:main}.]  
Let $\phi$ be a faithful representation of $\mathscr{S}^\flat$ 
as functions on some set $X$, in which the zero element $S$ is represented 
as the empty function.
Say $\mathscr{S}^\flat$ is isomorphic to $\mathscr{T}$ with universe $T$ and operations
$\compo$, $\D$, $\R$ and $\meet$ (when meet is in the signature)
interpreted as composition, domain, range and intersection of functions.

We may further assume that $X$ is the union of the domains of the functions in $\phi(S^\flat)$.
Define $\phi^\sharp$ on $S$ by 
\begin{equation}\label{eq:phisharp}
\phi^\sharp(s):=\bigcup\{\phi(\upset(\Delta\compo s))\mid 0\notin\Delta\compo s\}.
\end{equation} 
First we show that $\phi^\sharp(s)$ is in fact a function on $X$ for every $s\in S$.  
For this it suffices to show that if $\Delta_1\neq \Delta_2$ and 
$0\not\in \Delta_1\compo s, \;\Delta_2\compo s$ then 
$\phi(\upset(\Delta_1\compo s))$ has disjoint domain from 
$\phi(\upset(\Delta_2\compo s))$.  
Observe that $\Delta_1\compo\Delta_2$ contains $0$,
since there is $\alpha\in \D(S)$ such that $\alpha$ and $\A(\alpha)$
are in the symmetric difference of $\Delta_1$ and $\Delta_2$
and $0=\alpha\compo\A(\alpha)$.
Then, as composition of restrictions of the identity coincides with intersection
in representable algebras 
(note that this holds even when meet is not in the signature),
\begin{align*}
\D(\phi(\upset(\Delta_1\compo s)))\cap\D(\phi(\upset(\Delta_2\compo s)))&=
\D(\phi(\upset(\Delta_1\compo s)))\compo\D(\phi(\upset(\Delta_2\compo s)))\\
&=\phi(\Dd(\upset(\Delta_1\compo s)))\compo\phi(\Dd(\upset(\Delta_2\compo s)))\\
&=\phi(\Dd(\upset(\Delta_1\compo s))\mathop{*}\Dd(\upset(\Delta_2\compo s)))\\
&=\phi(\upset(\Delta_1\compo\Delta_2))\\
&=\phi(S)\\
&=\varnothing
\end{align*}
as desired.

Next we show that if $s\neq t$ in $S$ then $\phi^\sharp(s)\neq \phi^\sharp(t)$. 
Without loss of generality, assume that $s\not\leq t$.  
We show that there is an ultrafilter $\Delta$ of $D(S)$ such that 
$\upset(\Delta\compo s)$ is an ultrasubset and 
$\upset(\Delta\compo s)\neq \upset(\Delta\compo t)$.  
As $\phi$ is a faithful representation, we then have 
$\phi(\upset(\Delta\compo s))\neq \phi(\upset(\Delta\compo t))$. 

We claim that the set $I=\set{\alpha\in D(S)\mid\alpha\compo s\leq\alpha\compo t}$ 
is an ideal in the boolean algebra of domain elements (Lemma~\ref{lem:BA}).   
For downward closure, suppose $\alpha\in I$ and $\alpha_0\leq\alpha$.    
Then $\alpha_0\compo s\leq\alpha\compo s\leq\alpha\compo t$, so 
$\alpha_0\compo s \leq\alpha_0\compo\alpha\compo t=\alpha_0\compo t$ by 
Lemma~\ref{lem:BA}\eqref{part:mon}, so $\alpha_0\in I$. 
Now suppose $\alpha, \beta\in I$.   We have $\alpha\compo t\geq\alpha\compo s$ and
$\beta\compo t\geq\beta\compo s$, hence 
$\alpha\compo s\leq (\alpha\vee\beta) \compo t\geq \beta\compo s$
by Lemma~\ref{lem:BA}\eqref{part:vee2}.  
By Lemma~\ref{lem:BA}\eqref{part:vee} it follows that 
$(\alpha\vee\beta)\compo t\geq(\alpha\vee\beta)\compo s$ so $\alpha\vee\beta \in I$, as required.  
So $I$ is an ideal.
Clearly also, it avoids $\D(s)$.  
Thus we may extend the principal filter of $\D(s)$ in $D(S)$ 
to an ultrafilter $\Delta$ disjoint from $I$.
Now $0\notin\Delta\compo s$ because if $\alpha\compo s=0$, 
then trivially $\alpha\compo s\leq \alpha \compo t$ so that 
$\alpha$ would be in $I$.  
Thus $\upset(\Delta\compo s)$ is an ultrasubset. 
Moreover, it is clear that $t\notin\upset(\Delta\compo s)$ 
(as then we would have $\beta\in \Delta$ with $\beta\compo s\leq \beta\compo t$,
contradicting the choice of $\Delta$).  
So by Lemma~\ref{lem:allultra}, it follows that either $\upset(\Delta\compo t)$ 
is equal to $S$ or it is disjoint from $\upset(\Delta\compo s)$.  
In either case there are $x,y\in X$ with $(x,y)$ related by 
$\phi(\upset(\Delta\compo s))$ but not by $\phi(\upset(\Delta\compo t))$.  
Since we showed that for $\Gamma\ne\Delta$,
$\phi(\upset(\Delta\compo s))$ and $\phi(\upset(\Gamma\compo t))$
have disjoint domains, it follows that
$(x,y)$ is related by $\phi^\sharp(s)$ but not by $\phi^\sharp(t)$ as required.

We show that $\phi^\sharp$ is a homomorphism.
We start with the preservation of~$\compo$.    
For all $s, t\in S$,
\begin{align*}
\phi^\sharp(s)\compo\phi^\sharp(t)&=
\left(\bigcup_{0\notin\Delta;s} \phi\left(\upset(\Delta\compo s)\right)\right)
\compo \left(\bigcup_{0\notin\Gamma;t} 
\phi\left(\upset(\Gamma\compo t)\right)\right)\\
&=\bigcup_{0\notin\Delta;s} \bigcup_{0\notin\Gamma;t}
\phi\left(\upset(\Delta\compo s)\right)\compo
\phi\left(\upset(\Gamma\compo t)\right)\\
&=\bigcup_{0\notin\Delta;s,\; 0\notin\Gamma;t}
\phi\left(\upset(\Delta\compo s)
\mathop{*}\upset(\Gamma\compo t)\right)\\
&=\bigcup_{0\notin\Delta;s,\; 0\notin\Gamma;t}
\phi\left(\upset((\Delta\compo s)\compo(\Gamma\compo t))\right)\\
&=\bigcup_{0\notin\Delta;s;t}\phi\left(\upset(\Delta\compo s\compo t)\right)\\
&=\phi^\sharp(s\compo t).
\end{align*}
Note that the penultimate equality uses the fact that 
$\upset(\Delta\compo s\compo \Gamma\compo t)=\upset(\Delta\compo s\compo t)$ 
when $0\notin \Delta\compo s\compo \Gamma\compo t$ 
(by Lemma~\ref{lem:ultrasubsetfacts}), and
$\phi(\upset(\Delta\compo s\compo \Gamma\compo t))=\phi(S)=\varnothing$
when $0\in \Delta\compo s\compo \Gamma\compo t$.

Even though the operation $\D$ is only a derived operation in 
$(S,\compo,\A,\R)$, it is convenient to verify that
it is preserved before showing preservation of $\A$.  
We have 
\begin{align*}
\phi^\sharp(\D(s))&=
\bigcup_{0\notin\Delta;\D(s)} \phi\left(\upset(\Delta\compo \D(s))\right)\\
&=\bigcup_{0\notin\Delta;s} \phi(\upset\Delta)\\
&=\bigcup_{0\notin\Delta;s} \phi(\Dd(\upset(\Delta;s)))\\
&=\D\left(\bigcup_{0\notin\Delta;s} \phi(\upset(\Delta\compo s))\right)\\
&=\D(\phi^\sharp(s))
\end{align*}
as required.
The second equality uses the fact that 
$0\notin\Delta\compo \D(s)$ if and only if $0\notin \Delta\compo s$.

Now we check preservation of $\A$.  
As $\A(s)\compo s=0$ and $\phi^\sharp(0)=\phi(S)=\varnothing$ and $\compo$ is preserved, 
we must have $\dom(\phi^\sharp(\A(s)))\subseteq X\backslash \dom(\phi^\sharp(s))$. 
Let $x\in X\backslash \dom(\phi^\sharp(s))$, that is,  
$(x,x)\notin \D(\phi^\sharp(s))$.  
Now as $X$ is a union of the domains of elements of $\phi(S^\flat)$, 
it follows that there is an ultrafilter $\Delta$ of $\D(S)$ such that 
$x\in \dom(\phi(\upset\Delta))$, that is, $(x,x)\in\D(\phi(\upset\Delta))$.  
As $(x,x)\notin \D(\phi^\sharp(s))=
\bigcup\{\phi(\upset\Gamma)\mid 0\notin\upset(\Gamma\compo s)\}$, 
we must have $\upset(\Delta\compo s)=S$ showing that $\A(s)\in \Delta$.  
Then $(x,x)\in \bigcup_{\A(s)\in \Delta}\phi(\upset\Delta)=\phi^\sharp(\A(s))$ 
as required.

For preservation of $\R$ first use the fact that $s\compo \R(s)=s$ 
and the fact that $\compo$ is preserved and $\R(s)=\D(\R(s))$ 
to deduce that $\phi^\sharp(\R(s))$ is a restriction of the identity element 
whose domain contains the range of $\phi^\sharp(s)$.  
So $\R(\phi^\sharp(s))\subseteq \phi^\sharp(\R(s))$.  

For the other direction assume that $(y,y)\in \phi^\sharp(\R(s))$.  
We show that there is $x\in X$ with $(x,y)\in\phi^\sharp(s)$.  
Now, as $(y,y)\in \phi^\sharp(\R(s))$, there is $\Delta$ with 
$(y,y)\in \phi(\upset(\Delta\compo \R(s)))$.  
Hence $\R(s)\in\Delta$ and $\upset\Delta=\upset(\Delta\compo \R(s))$.  
Consider the filter in $D(S)$ generated by 
$F:=\{\A(\alpha)\mid \R(\alpha\compo s)\notin \Delta\}$.  
To show this is a proper filter, observe that 
$\R(\alpha\compo s)\vee\R(\beta\compo s)\refe{eq:antirange}
\R((\alpha\vee\beta)\compo s)$.   So if $\A(\alpha)$ and $\A(\beta)$ are in $F$
then (as $\Delta$ is a prime filter) we have 
$\R(\alpha\compo s)\vee\R(\beta\compo s)$ not in $\Delta$, whence
$\A(\alpha)\cdot\A(\beta)=\A(\alpha\vee\beta)\in F$.  
Let $\Gamma$ be any ultrafilter of $D(S)$ extending $F$.  
If $0\in \Gamma\compo s$ then there is $\alpha\in \Gamma$ with 
$0=\alpha\compo s$ so that $\R(\alpha\compo s)=0$, which would give 
$\A(\alpha)\in F\subseteq \Gamma$, a contradiction.  
So $\upset(\Gamma\compo s)$ is an ultrasubset.  
Now we show that $\R(\upset(\Gamma\compo s))\subseteq \Delta$.  
Consider any $\alpha\in\Gamma$. Then as $\A(\alpha)\notin F$, 
it follows that $\R(\alpha\compo s)\in\Delta$.  
Thus $\Rr(\upset(\Gamma\compo s))=\upset\Delta$.  
Now recall that $\upset\Delta=\upset(\Delta\compo \R(s))$.  
Since $\phi$ preserves $\Rr$ as range and 
$(y,y)\in\phi(\upset(\Delta\compo \R(s)))=\phi(\upset\Delta)=\phi(\Rr(\upset(\Gamma\compo s)))$, 
there must be $x\in X$ with $(x,y)\in \phi(\upset(\Gamma\compo s))$.  
Then $(x,y)\in \phi^\sharp(s)$ as required.

We have not used $\cdot$ to establish the preservation of $\compo,\A,\R$, 
so if $\meet$ is not present then the proof of Theorem~\ref{thm:main} is complete.  

Finally we must verify preservation of $\meet$ when it is present. 
Using that meet is interpreted as intersection in representable algebras and
Lemma~\ref{lem:ultrasubsetfacts} part~(iii),
\begin{align*}
\phi^\sharp(s)\meet\phi^\sharp(t)&=
\phi^\sharp(s)\cap\phi^\sharp(t)\\
&=\left(\bigcup_{0\notin\Delta;s}\phi(\upset(\Delta\compo s))\right)\cap
\left(\bigcup_{0\notin\Gamma;t}\phi(\upset(\Gamma\compo t))\right)\\
&=\bigcup_{0\notin\Delta;s,\;0\notin\Gamma;t}\phi(\upset(\Delta\compo s))\cap
\phi(\upset(\Gamma\compo t))\\
&=\bigcup_{0\notin\Delta;s,\;0\notin\Gamma;t}\phi(\upset(\Delta\compo s))\wedge
\phi(\upset(\Gamma\compo t))\\
&=\bigcup_{0\notin\Delta;s,\; 0\notin\Gamma;t}\phi\left(\upset((\Delta\compo s)\meet
(\Gamma\compo t))\right).
\end{align*}
Using Lemma~\ref{lem:ultrasubsetfacts} part~(iv),
\begin{equation*}
0\notin (\Delta\compo s)\meet (\Gamma\compo t)
\mbox{ if and only if } \Delta=\Gamma \mbox{ and }
0\notin \Delta\compo (s\meet t)
\end{equation*}
whence
\begin{equation*}
\bigcup_{0\not\in\Delta;s,\;\Gamma;t}\phi\left(\upset((\Delta\compo s)\meet
(\Gamma\compo t))\right)=
\bigcup_{0\not\in\Delta;(s\cdot t)}\phi(\upset(\Delta\compo(s\meet t)))=
\phi^\sharp(s\meet t)
\end{equation*}
as desired.
\end{proof}

\section{Equational Theory}
Let $\tau\subseteq \{\compo, \cdot, \sqcup, \D, \R, \A, \Fix, 0, 1'\}$.  
A \emph{term} is either $0$, a single variable symbol, or recursively 
$(t_1\compo t_2), (t_1\cdot t_2), (t_1\sqcup t_2), \D(t), \R(t), \A(t)$ or $\Fix(t)$ 
(if the relevant operations are in $\tau$), where $t_1, t_2, t$ are terms. 
We write $t(\bar x)$ for a term using only variables in the $n$-tuple of variables $\bar x$, 
but not necessarily all of them.  For $\SS\in\F(\tau)$ and a $n$-tuple $\bar a$ of elements
(functions) in $\SS$, we interpret $t(\bar a)$ as a partial function over the base of $\SS$ 
as we explained in section 1.   An \emph{equation} has the form $t(\bar x)=s(\bar x)$, 
where $\bar x$ is an $n$-tuple of variables and $t(\bar x), s(\bar x)$ are terms.    
It is valid if for every $\SS\in\F(\tau)$ and every $n$-tuple $\bar a$ of elements of $\SS$, 
the partial functions $t(\bar a)$ and $s(\bar a)$ are identical. 
(For the sake of simplicity, we will call $t(\bar a)$ etc.\ terms as well.)

\begin{theorem}  \label{thm:eq th}
Let $\tau\subseteq \{\compo, \meet, \sqcup,\D, \R, \A, \Fix,0,1'\}$ and
$\Sigma_\tau$ be the set of equations valid over $\F(\tau)$.
\begin{enumerate}
\item 
$\Sigma_\tau$ is  {\bf co-NP}.
\item  
If $\{\compo, \A\}\subseteq\tau$ then $\Sigma_\tau$ is {\bf co-NP-complete}.
\end{enumerate}
\end{theorem}
\begin{proof}
For the first part, let $\bar x$ be an $n$-tuple of variables and suppose the equation 
$u(\bar x)=v(\bar x)$ is not valid over $\F(\tau)$.  Then there is $\mathscr S\in\F(\tau)$ 
and $u(\bar a)\neq v(\bar a)$, for some $n$-tuple $\bar a=(a_0, a_1, \ldots, a_{n-1})$ 
of elements in $\SS$. Let $X$ be the base of~$\SS$. 
The two distinct functions $u(\bar a), v(\bar a)$ must disagree on at least one point  of $X$ 
(disagreement includes the possibility that one partial function is defined at this point
but not the other).    
The plan is to find a finite subset $Y\subseteq X$ of size no more than the sum of the 
lengths of the terms $u(\bar a), v(\bar a)$ such that the equation already fails in an 
algebra $\SS\restr Y$ with base set $Y$,  obtained from $\SS$ by restricting all functions 
and operations to $Y\times Y$.  The first part of the theorem will follow, 
since a non-deterministic machine could check to see if the equation $u(\bar x)=v(\bar x)$ 
failed in any algebra of functions on a base of size at most $|u(\bar a)|+|v(\bar a)|$
in quadratic time.

Let us clarify what we mean by $\SS\restr Y$.   
For any $Y\subseteq X$ we let $\SS\restr Y$ be the algebra of functions 
$\set{f\cap (Y\times Y):f\in\SS}$ with the operations 
$\{\compo, \cdot, \D, \R, \A_Y, \Fix, \varnothing, 1'_Y\}$ 
(if the relevant operation is included in $\tau$), where 
\begin{align*}
1'_Y&=\set{(y, y)\mid y\in Y}\\ 
\A_Y(f)&=\set{(y, y)\mid y\in Y,\;  \neg(\exists z\in Y)(y, z)\in f}
\end{align*}
(the other operations do not need relativizing, since $\SS\restr Y$ 
is already closed under composition, intersection, domain, range and fixset).  

We consider terms $t(\bar a)$ constructed from $a_0, a_1, \ldots, a_{n-1}$ 
using operations in $\tau$.
Such a term is directly  evaluated in $\SS$ using the set-theoretically defined operations 
of $\SS$.  We write $t(\bar a)\restr Y$ for that function in $\SS\restr Y$ obtained 
from the tuple of elements 
$(a_0\cap(Y\times Y), a_1\cap(Y\times Y),\ldots, a_{n-1}\cap(Y\times Y))$, 
using the (relativized) constants and operations of $\SS\restr Y$.  
Warning: it is not in general true that $t(\bar a)\restr Y=t(\bar a)\cap (Y\times Y)$, 
for example let $a$ be the function $\set{(y, z)}$ (some $y\neq z\in X$) and let $Y=\set y$, 
then $a\restr Y$ is empty and hence $(\D (a))\restr Y=\varnothing$, 
whereas $\D (a)\cap (Y\times Y)=\set{(y, y)}$.

For $x\in X$ and any term $t(\bar a)$ using only the elements $a_0,\ldots, a_{n-1}$ of $\SS$,
we construct a finite subset $\Sigma(t(\bar a), x)$ of $X$.  
When applied to the terms $u$ and $v$ selected above 
(and for which $u(\bar{a})$ and $v(\bar{a})$ disagree at some point), 
this will provide a small finite subset of $X$ in which $u(\bar{a})\neq v(\bar{a})$ 
remains witnessed.  The construction is by induction on the complexity of terms.
\begin{align*}
\Sigma(0, x) &=\set x\\
\Sigma(1', x)&=\set x\\
\Sigma(a_i, x)&=\left
\{\begin{array}{ll}     \set{x, a_i(x)}&\mbox{if $a_i$ is defined on $x$}\\ 
\set x&\mbox{otherwise}\end{array}\right. (i<n)\\
\Sigma(\Fix(t(\bar a)), x)&=\set x\\
\Sigma(\D(t(\bar a)), x)&=\Sigma(\A(t(\bar a)), x)=\Sigma(t(\bar a), x)\\
\Sigma(\R(t(\bar a)), x)&=\left
\{\begin{array}{ll}\Sigma(t(\bar a), y) &\mbox{some arbitrary $y$ with $t(\bar a)(y)=x$}\\ 
\set x&\mbox{if no such $y$ exists}\end{array}\right.\\
\Sigma(s(\bar a)\compo t(\bar a), x)&=\left\{\begin{array}{ll}
\Sigma( s(\bar a), x)\cup\Sigma(t(\bar a), s(\bar a)(x))&
\mbox{if $(s(\bar a)(x))$ is defined}\\ 
\Sigma(s(\bar a), x)&\mbox{otherwise}\end{array}\right.\\
\Sigma(s(\bar a)\cdot t(\bar a), x)&=\Sigma(s(\bar a), x)\cup\Sigma(t(\bar a), x)\\
\Sigma(s(\bar a)\sqcup t(\bar a), x)&=\Sigma(s(\bar a), x)\cup\Sigma(t(\bar a), x)
\end{align*}
Clearly, the size of $\Sigma(t(\bar a), x)$ is no bigger than twice the length of the term 
$t(\bar a)$.  We may drop the $\bar a$ and refer to a term simply as $t$.

Observe that $x\in\Sigma(t,x)$ and that $t(x)\in \Sigma(t,x)$ whenever $t$ is defined on $x$.

We claim that, for any $Y$ with $\Sigma(t, x)\subseteq Y\subseteq X$ and $x\in \Sigma(t,x)$,
\begin{equation}\label{eq:restr}
t(x)= t\restr Y(x)
\end{equation}
including that $t(x)$ is defined iff $t\restr Y(x)$ is defined.

For the base cases,  observe that $0$ is the empty function in both $\SS$ and $\SS\restr Y$,
$1'(x)=1'_Y(x)=x$, since $x\in \Sigma(1', x)\subseteq Y$, and for $t=a_i$ (some $i<n$), 
$a_i$ is defined at $x$ iff $a_i\restr Y$ is defined at $x$ 
(since $\set{x, a_i(x)}\subseteq\Sigma(a_i, x)\subseteq Y$) and if defined they are equal.

Next suppose $t=r\cdot s$ (some terms $r,s$).  Then $r\cdot s$ is defined at $x$ iff 
$r(x)=s(x)$ ($=y$, say) iff $r\restr Y(x)=s\restr Y(x)=y$ (inductively) iff 
$(r\cdot s)\restr Y(x)=y$.    
Similarly $r\sqcup s$ is defined at $x$ (say $(r\sqcup s)(x)=y$) iff either $r(x)=y$ or $r$ is not defined at $x$ but $s(x)=y$.  This holds if and only if  $r\restr Y(x)=y$ or $r\restr Y$ is not defined at $x$ but $s\restr Y(x)=y$ (inductively) iff $(r\sqcup s)\restr Y(x)=y$.
The case $t=\Fix(s)$ (some term $s$) is also similar (after all $\Fix(v) = 1'\cdot v$).

Consider the case  $t=\A (s)$ (some term $s$).  
We have: $(\A (s))(x)$ is defined and equal to $x$ iff $s(x)$ is not defined iff 
(by induction) $s\restr Y(x)$ is not defined iff $(\A s)\restr Y(x)$ is defined and equal to $x$, 
since $\Sigma(s, x)=\Sigma(\A (s), x)\subseteq Y\subseteq X$. 
The case $t=\D(s)$ is  similar.  

Now let $t=r\compo s$ (some terms $r, s$).  If $r\compo s$ is defined at $x$ then 
$(r\compo s)(x)=s(r(x))$ and $\set{x, r(x), s(r(x))}\subseteq\Sigma(r\compo s, x)\subseteq Y$.  
So $r\restr Y(x)=r(x)$ is defined and $s\restr Y(r(x)) =s(r(x))$ is also defined, 
hence $(r\compo s)\restr Y(x)=s(r(x))$ is defined.  
Conversely, if $(r\compo s)\restr Y(x)$ is defined then $r(x)=r\restr Y(x)$ is defined 
and $s\restr Y(r(x))=s(r(x))$ is also defined, so $(r\compo s)(x)=(r\compo s)\restr Y(x)$ is 
also defined.

Finally consider the case $t=\R (s)$ (some term $s$).  If $\R (s)$ is defined at $x$ 
then by definition of $\Sigma(\R (s), x)$ there is $y\in\Sigma(\R (s), x)$ with $s(y)=x$
and $\Sigma(\R (s), x)=\Sigma(s, y)$.  By induction, $s\restr Y(y)=x$, so $(\R (s))\restr Y(x)=x$.
Conversely, if $(\R (s))\restr Y(x)$ is defined then there is $y\in Y$ with $s\restr Y(y)=x$, 
hence $s(y)=x$ and $(\R (s))(x)=x$, as required.  This proves the claim~\eqref{eq:restr}.

Now recall that $u(\bar x)=v(\bar x)$ was an equation failing in $\SS$ under some assignment  
mapping the tuple of variables $\bar x$ to the tuple of elements $\bar a$ of $\SS$ such that the 
two functions $u(\bar a), v(\bar a)$ are distinct.  That is, $u(\bar a)$ and $v(\bar a)$ disagree 
at some point, say $u(\bar a)$ disagrees with $v(\bar a)$ at $x_0\in X$.  
Let $Y=\Sigma(u(\bar a), x_0)\cup\Sigma(v(\bar a), x_0)$.  By the claim, 
$u(\bar a)\restr Y(x_0)$ agrees with $u(\bar a)(x_0)$ (both defined and equal or neither defined) 
and $v(\bar a)(x_0)$ agrees with $v(\bar a)\restr Y(x_0)$, hence $u(\bar a)\restr Y$ 
disagrees with $v(\bar a)\restr Y$ at $x_0$.  So  the equation $u(\bar x)=v(\bar x)$ fails 
in $\SS_Y$ under the assignment mapping $x_i$ to $a_i\cap(Y\times Y)$, for $i<n$.

Thus an equation $u(\bar x)=v(\bar x)$ fails to be valid over $\F(\tau)$ if and only if it fails 
in some algebra $\mathscr{S}$ of functions on a base $Y$ whose size is linear in terms of the 
equation.  
It follows that we can test the failure of equations by non-deterministically generating 
a labelled directed graph of this size and verifying if the equation fails.

For the second part, we reduce the validity problem for propositional formulas 
(see~\cite[\S A9]{garjoh} for example)
to membership of $\Sigma_\tau$.    
We may assume that our propositional language includes only the connectives $\neg, \wedge$.
Take a propositional formula $\phi$ and replace each proposition $p$ by $\D(f_p)$ 
for some function symbol $f_p$ unique to $p$ and 
replace $\neg$ and $\wedge$ by $\A$ and $\compo$, respectively, to obtain a term $\phi^*$.  
The required reduction maps $\phi$ to the equation $\phi^*=1'$.  
This reduction is correct  by Lemma~\ref{lem:BA}.
\end{proof}

As mentioned in the introduction, one cannot hope for such a result when maximal iterate 
is included in the signature, as the equational theory is known to be 
$\Pi_1^1$-hard~\cite{goljac}, at least in signatures containing $\{\compo,\Fix,\A,^\uparrow\}$.

\section{Finite representation for domain, range and composition}
We now revisit Schein's representation for the signature $(\compo,\D, \R)$ 
and present an identification 
that yields a finite representation in the case of finite algebras.  
Using the  results from previous sections of this article, we then obtain 
a finite representation for finite representable algebras in  various meet-free signatures
extending $(\compo, \D, \R)$, see Theorem~\ref{thm:last} below.

Let us recall Schein's representation for  appropriate algebras of the signature 
$\tau=\{\compo,\D,\R\}$.
Let ${\mathscr S}=( S,\compo,\D,\R)$, possibly with $0$, 
be an associative algebra satisfying \eqref{eq:leftid}--\eqref{eq:Dtwisted} 
and \eqref{eq:rightid}--\eqref{eq:DR}.  
The base of  Schein's representation consists of certain finite sequences
of  elements from $S\backslash\{0\}$.

For $n\geq -1$, a sequence $(a_0,b_0,\dots,a_{n},b_{n},a_{n+1})\in S^{2n+3}$ is 
\emph{permissible} if $\R(a_i)=\R(b_i)$ and $\D(b_i)=\D(a_{i+1})$, for all $i\leq n$.
A permissible sequence $(a_0,b_0,\dots,a_n,b_n,a_{n+1})$ reduces to a sequence 
$(a_0,b_0,\dots,a_n\compo x)$ if $b_n\compo x=a_{n+1}$.  See Figure \ref{fig:reduction}.
\begin{figure}
\begin{center}
\begin{tikzpicture}
	 \draw [->] (0,0) -- (0.95,1);
	 \draw [->] (2,0) -- (1.05,1);
	 \draw [->] (2,0) -- (2.95,1);
	 \draw [->] (3.2,.8) -- (3.05,1);
	 
	\node at (0.2,0.6) {$a_0$};
	\node at (1.75,0.65) {$b_1$};
	\node at (2.3,0.6) {$a_1$};
	
	 \draw [->] (5,0) -- (5.95,1);
	 \draw [-] (5,0) -- (4.8,0.2);
	 \draw [->] (7,0) -- (6.05,1);
	 \draw [->] (7,0) -- (7.95,1);
	 \draw [->] (6,1.1) -- (7.9,1.1);
	 
	\node at (5.2,0.6) {$a_n$};
	\node at (6.75,0.65) {$b_n$};
	\node at (8.2,0.6) {$a_{n+1}$};
	\node at (7,1.3) {$x$};
	
	  \draw [fill] (0,0) circle [radius=0.1]; 
	  \draw [fill] (1,1.1) circle [radius=0.1]; 
	  \draw [fill] (2,0) circle [radius=0.1]; 
	  \draw [fill] (3,1.1) circle [radius=0.1]; 
	  \draw [fill] (5,0) circle [radius=0.1]; 
	  \draw [fill] (6,1.1) circle [radius=0.1]; 
	  \draw [fill] (7,0) circle [radius=0.1]; 
	  \draw [fill] (8,1.1) circle [radius=0.1]; 
	     
	\node at (4,0.6) {$\dots$}; 
 \end{tikzpicture}\end{center}
 \caption{\label{fig}A permissible sequence $(a_0,b_1,a_1,\dots,a_n,b_n,a_{n+1})$ 
 where $b_n\compo x=a_{n+1}$.  This sequence can be reduced to 
 $(a_0,b_1,a_1,\dots,a_n\compo x)$.}\label{fig:reduction}
 \end{figure}
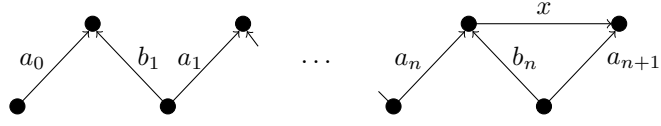

The sequence $(a_0,b_0,\dots, b_{n-1}, a_n\compo x)$ is itself permissible 
because of the following. 
Notice that $\D(a_n\compo x)\leq \D(a_n)=\D(b_{n-1})$ by \eqref{eq:Dorder}.  
Also, 
$b_n\compo \D(x)\refe{eq:Dtwisted}\D(b_n\compo x)\compo b_n=
\D(a_{n+1})\compo b_n=\D(b_n)\compo b_n=b_n$.  
Thus $\D(x)\geq \R(b_n)=\R(a_n)$.  So 
$\D(a_n\compo x)\refe{eq:leftcong}\D(a_n\compo \D(x))=\D(a_n)=\D(b_{n-1})$, 
showing that $(a_0,b_0,\dots,a_n\compo x)$ is permissible.  
Reduction is easily shown to be unique using the implication~\eqref{eq:schein}: 
$b\compo x=b\compo y\Rightarrow \R(b)\compo x=\R(b)\compo y$.

A permissible sequence is \emph{reduced} if no reductions are possible.  
Given a sequence $\alpha$, we denote the unique reduced sequence obtained 
by reducing $\alpha$ by $\nf(\alpha)$, the \emph{normal form} of $\alpha$.
Schein defines a representation $\theta$ on sequences in normal form as follows.
For an element $a$ of $S$, the representation $a^\theta$ of $a$ will be a function
with the following domain
$$
\dom(a^\theta)=\{(a_0,b_0,\dots,a_n,b_n,a_{n+1})\mid \D(a)\geq\R(a_{n+1})\}
$$ 
and then 
$$
a^\theta(a_0,b_0,\dots,a_n,b_n,a_{n+1})=\nf(a_0,b_0,\dots,a_n,b_n,a_{n+1}\compo a).
$$

It can be shown, with the aid of Figure~\ref{fig},  for any element $a$
and permissible sequence $(a_0,b_0,\dots,a_{n+1})$, that $a_{n+1}\compo \D(a)=a_{n+1}$ 
if and only if $\nf(a_0,b_0,\dots,a_{n+1})$  is in the domain of $a^\theta$, 
and that in this case, 
$\nf(a_0,b_0,\dots,a_{n+1}\compo a)=a^\theta(\nf(a_0,b_0,\dots,a_{n+1}))$,  
(for example, see~\cite[Lemmas~4.7,~4.8]{jacsto:DR} for full details).  
Hence
\begin{equation}\label{eq:nf}
a^\theta(\nf(a_0, b_0,\ldots, a_{n+1}))=\nf(a_0, b_0,\ldots, b_n, a_{n+1}\compo a)
\end{equation}
for any permissible sequence $(a_0, \ldots, a_{n+1})$ and $a\in S$ 
such that $\D(a)\geq\R(a_{n+1})$.

This representation preserves $\compo$, $\D$ and $\R$ (Schein~\cite{sch:DR})
as well as $0$ as $\varnothing$ (if $0=\D(0)$ is present). 
For example, to check $\R$, the reduced  permissible sequence $(a_0, \ldots, a_{n+1})$ 
is in the domain of the function $(\R(a))^\theta$ if and only if 
$\R(a_{n+1})\leq\D(\R(a))=\R(a)$ and in that case 
$(\R(a))^\theta(a_0, \ldots, a_{n+1})=(a_0, \ldots, a_{n+1})$. 
Also $a^\theta$ is defined on  
$\nf(a_0, \ldots, a_{n+1}, \; a\compo \R(a_{n+1}), \; \D(a\compo\R(a_{n+1})))$ 
and 
\begin{eqnarray*}
\lefteqn{a^\theta(\nf(a_0, \ldots, a_{n+1},\; a\compo\R(a_{n+1}), \; \D(a\compo\R(a_{n+1}))))}\\
&=&\nf(a_0, \ldots, a_{n+1},\; a\compo \R(a_{n+1}),\;\D(a\compo\R(a_{n+1}))\compo a)\\
&=&(a_0, \ldots, a_{n+1})
\end{eqnarray*}
using a reduction with $x=\R(a_{n+1})$, since 
$a\compo\R(a_{n+1})\compo\R(a_{n+1})=\D(a\compo\R(a_{n+1}))\compo a$.

For the other direction assume that
$\nf (a_0, \ldots, b_n, a_{n+1}\compo a)$ is in the range of $a^\theta$.
Since $(a_0, \ldots, b_n, a_{n+1}\compo a)=(a_0, \ldots, b_n, a_{n+1}\compo a\compo \R(a))$, 
applying $(\R(a))^\theta$ fixes $\nf(a_0, \ldots, b_n, a_{n+1}\compo a)$ (use~\eqref{eq:nf}).
Therefore 
$\R(a)^\theta$ is equal to the identity restricted to the range of the 
function $a^\theta$, as required.  
Similarly, $\D$ is represented correctly by~$\theta$.

This representation also preserves~$\meet$ if it is present and the 
appropriate axioms are satisfied (Jackson and Stokes~\cite{jacsto:DR}).  
Except in trivial cases there are infinitely many reduced permissible words over $S$, 
and then this representation is over an infinite domain, even when $S$ is finite.   
We now observe a further identification that for finite $S$ will produce a 
faithful representation over a finite domain for 
$\{\compo,\D,\R,0\}$, though not in general for~$\meet$.

\begin{theorem}\label{thm:dr}
Let $\SS$ be an associative $\set{\compo,\D, \R}$-algebra satisfying 
\eqref{eq:leftid}--\eqref{eq:Dtwisted} and \eqref{eq:rightid}--\eqref{eq:DR}. 
Then $\SS$ has a representation on a base of size at most $|S|^{1+|S|}$ 
and at most $(|S|-1)^{|S|}$ if there is a $0$ element with $\D(0)=0$.
\end{theorem}

\begin{proof} 
(The reader may wish to follow this argument in conjunction with a reading of 
Example~\ref{eg:ab}, where some of the constructions of the proof are given for a simple example.)
Let us say that the \emph{address} of a sequence $(a_0,b_0,\dots,a_n,b_n,a_{n+1})$ is $a_{n+1}$,
and denote the address of sequence $\bar a$ by $\add(\bar a)$.  
The \emph{view} of a reduced permissible sequence $\bar a$ is the set 
\begin{align*}
\view(\bar a)
&=\{(x,y)\in S\times S\mid \bar a\in\dom(x^\theta)\mbox{ and }\add(x^\theta(\bar a))=y\}\\
&=\{(x,\add(x^\theta(\bar a)))\mid \add(\bar a)\compo \D(x)=\add(\bar a)\}.
\end{align*}
A view is a partial function from $S$ to $S$, hence when $S$ is finite, 
the number of views is at most $|S|^{1+|S|}$.  
If there is $0=\D(0)$ then we may replace $S$ by $S\backslash\{0\}$, 
giving at most $(|S|-1)^{|S|}$ views.

Let $\bar a=\nf(\bar a)=(a_0,b_0,\dots,a_n,b_n,a_{n+1})$ be a reduced sequence.
Then 
\begin{align*}
(\R(\add(\bar a)))^\theta(\bar a)&=
(\R(a_{n+1}))^\theta(a_0,b_0,\dots,a_n,b_n,a_{n+1})\\
&=\nf(a_0,b_0,\dots,a_n,b_n,a_{n+1}\compo\R(a_{n+1}))\\
&=\nf(a_0,b_0,\dots,a_n,b_n,a_{n+1})\\
&=\bar a.
\end{align*}
Hence $(\R(\add(\bar a)),\add(\bar a))\in\view(\bar a)$.
In particular, for a sequence $(c)$ of length $1$,
we have $(\R(c),c)\in\view(c)$ and
$\view(c)=\{(x,c\compo x)\mid \R(c)\leq\D (x)\}$.

Define an equivalence relation $\equiv$ on reduced permissible sequences by 
\begin{equation}
\bar a\equiv \bar b\mbox{ if }
\view(\bar a)=\view(\bar b).\label{eq:equiv}
\end{equation}
Since $(\R(\add(\bar a)),\add(\bar a))\in\view(\bar a)$, we have that
\[
\view(\bar a)=\view(\bar b)\mbox{ implies }
\add(\bar a)=\add(\bar b).
\]
In particular this shows that distinct permissible sequences of length $1$ lie 
in distinct equivalence classes modulo $\equiv$.
Also, if $S$ is a finite set then there are only finitely many possible views, 
so that $\equiv$ has finitely many blocks.
We now show that the functions $x^\theta$ preserve this equivalence relation, 
and that domains of functions $x^\theta$ are unions of blocks of the equivalence relation.  
Thus if $X$ denotes the set of reduced permissible words, 
then ${\mathscr S}$ is also represented on $X/{\equiv}$
by the map 
\begin{equation}\label{eq:Theta}
s^\Theta=(s^\theta/{\equiv}).
\end{equation}

Fix $x\in S$ and assume that $\bar a\equiv \bar b$.  
By the definition of view we have that $\bar a\in \dom(x^\theta)$ if and only if 
$\bar b\in \dom(x^\theta)$.  
This shows that the domain of $x^\theta$ is a union of blocks.  

Next we show that $x^\theta(\bar a)$ is equivalent modulo $\equiv$ to $x^\theta(\bar b)$.  
Let $(z,c)$ be in the view of $x^\theta(\bar a)$.  
So $x^\theta(\bar a)\in\dom(z^\theta)$ and $\add(z^\theta(x^\theta(\bar a)))=c$.  
That is, $\bar a\in\dom((x\compo z)^\theta)$  and $\add((x\compo z)^\theta(\bar a))=c$.  
Hence 
$((x\compo z),c)=((x\compo z), \add((x\compo z)^\theta(\bar a)))\in\view(\bar a)=\view(\bar b)$
as $\bar a$ and $\bar b$ have the same view.  
So $\bar b\in\dom((x\compo z)^\theta)$ and  $\add((x\compo z)^\theta(\bar b))=c$.
Thus  $x^\theta(\bar b)\in\dom(z^\theta)$ and $\add(z^\theta(x^\theta(\bar b)))=c$,
that is,  $(z,c)\in\view(x^\theta(\bar b))$, as required.

The faithfulness of the representation $\Theta$ of ${\mathscr S}$ on $X/{\equiv}$
follows from the faithfulness of the representation $\theta$ of ${\mathscr S}$ on $X$,
since this is witnessed over sequences of length $1$
and distinct length $1$ sequences are never equivalent.
\end{proof}

The following basic example may aid the reader.
\begin{example}\label{eg:ab}
Consider the 5-element algebra in the signature $\{0, \compo, \cdot,\D,\R\}$ 
consisting of elements $\{0,a,b,d,r\}$ with $0$, $d$ and $r$ domain elements 
and with $\D (a)=\D (b)=d$ and $\R (a)=\R (b)=r$.  
All elements are disjoint \up(meeting to $0$ under $\meet$\up) 
and the only nonzero products with respect to $\compo$ are those forced 
by the usual properties of $\D$ and $\R$, for example $d\compo a\compo r=a$. 
Note that
\[
\D(x)=\left\{  \begin{array}{ll}   r&\mbox{if }x=r\\
d & x\neq r\end{array}\right. \mbox{ and } 
\R(x)=\left\{\begin{array}{ll}d&\mbox{if }x=d\\ r&
x\neq d\end{array}\right.
\]
for non-zero $x$.
So, a sequence $(x_0, x_1, \dots, x_{2n})$  \up(where $n\geq 0$\up) of non-zero elements 
is permissible if 
\up{(i)} $(x_{2i}, x_{2i+1})\in\set{(d, d)}\cup\set{(w, z)\mid d\not\in\set{w, z}}$ and 
\up{(ii)} $(x_{2i+1}, x_{2i+2})\in\set{(r, r)}\cup\set{(w, z)\mid r\not\in\set{w, z}}$, for $i<n$.
Their equivalence relation $\equiv$ has six blocks, namely,  the singleton $\{(r)\}$, 
the singleton $\{(d)\}$, and for $s=(a), (b), (a,d), (b,d)$, the set of all permissible sequences
ending with the string $s$.
\end{example}

We now extend this finite representation result to larger signatures including antidomain.

\begin{theorem}\label{thm:last}
Let $\tau$ be a signature with 
$\set{\compo, \D, \R}\subseteq\tau\subseteq\set{\compo, \D, \A, \R, \Fix,0,1'}$ 
or with 
$\set{\compo, \A, \R,\sqcup}\subseteq\tau\subseteq
\set{\compo, \D, \A, \R, \Fix,\sqcup, {}^\uparrow, 0,1'}$.  
Every finite, representable $\tau$-algebra $\SS=(S, \tau)$ is representable over a base of size 
at most  $|S|^{|S|+1}$. 
\end{theorem}

\begin{proof}
First assume that $\A\notin \tau$. So we are considering 
$\set{\compo, \D, \R}\subseteq\tau\subseteq\set{\compo, \D, \R, \Fix,0,1'}$.
The case $\tau=\set{\compo,\D, \R}$ is covered by Theorem~\ref{thm:dr}, 
and it is clear that the representation $\Theta$ defined in~\eqref{eq:Theta}   
correctly represents $0$ and $1'$ if one or both of these are present.
We now show that $\Theta$ also already represents the operation~$\Fix$ 
correctly if $\Fix\in\tau$.
In~\cite{jacsto:DR} it was shown that, because the laws 
$\Fix(x)\compo x=\Fix(x)$, $\D(\Fix(x))=\Fix(x)$ and $x\compo y=x\Rightarrow x\compo\Fix(y)=x$
are satisfied, Schein's representation will correctly represent $\Fix$. 
We now observe that $\Fix$ is still correctly represented after applying the 
identification~$\equiv$.  

Consider an element $x$ and a reduced sequence $\bar a$ in the domain of $x^\theta$ 
such that $x^\Theta$ fixes the $\equiv$-class $[\bar a]$ of $\bar a$,
that is,  $\bar a\equiv x^\theta(\bar a)$.  
Thus the view of $x^\theta(\bar a)$ is identical to that of $\bar a$.  
In particular, $x^\theta(\bar a)$ has the same address as $\bar a$, 
and is either strictly shorter than $\bar a$ or is identical to $\bar a$.  
Let $a$ denote the address of $\bar a$, and  let ${\bar a}'$ denote $x^\theta(\bar a)$ 
and ${\bar a}''$ denote $x^\theta({\bar a}')$ and so on. 
Each element of  $\bar a,{\bar a}',{\bar a}'',\ldots$ has the same view as $\bar a$ 
(so in particular, the same address, $a$).  
The sequence $\bar a,{\bar a}',{\bar a}'',\ldots$ is eventually constant, 
so eventually we arrive at some $\bar b\equiv\bar a$ that is fixed by $x^\theta$.  
Because the address of $\bar b$ is $a$, we then have $a\compo x=a$, 
which gives $a\compo \Fix(x)=a$.
So, for any reduced permissible sequence $\bar c\equiv\bar a$,   
we have $x^\theta(\bar c)=\bar c$, since $\add(\bar c)=a$.  Hence
\begin{eqnarray*}
([\bar a],[ \bar a])\in \fix(x^\Theta)&\iff&x^\Theta([\bar a])=[\bar a]\\
&\iff&x^\theta(\bar c)=\bar c\;\;\mbox{(all $\bar c\equiv\bar a$)}\\
&\iff&(\bar c, \bar c)\in \fix(x^\theta)\;\;\mbox{(all $\bar c\equiv\bar a$)}\\
&\iff&(\bar c, \bar c)\in (\fix(x))^\theta\;\;\mbox{(all $\bar c\equiv\bar a$)}\\
&\iff&([\bar a], [\bar a])\in (\fix (x))^\Theta.
\end{eqnarray*}
This completes the proof for cases where $\A\notin\tau$.  

Now assume $\A\in\tau$ and let $\SS\in\F(\tau)$ be finite and representable.  
We will temporarily ignore $\sqcup$ and ${}^\uparrow$ if they are present.
Since $\set{\compo, \A}\subseteq\tau$ the set $D(S)$ of domain elements of $\SS$ 
forms a boolean algebra and since $\SS$ is finite, this boolean algebra is atomic.  
Consider the set 
$S^{at}=\set{s\in \SS: \D(s)\mbox{ is an atom  of $D(S)$}}\cup\set0$.
It is clear that $S^{at}$ is closed under all the operations of $\tau$ except $\A$, 
so let $\SS^{at}$ be the algebra with universe $S^{at}$ and operations in 
$\tau\cap\set{\compo, \D, \R, \Fix,0,1'}$ inherited from $\SS$  
(in fact it is isomorphic to the algebra $\SS^\flat$ of Definition~\ref{defn:Sflat}).  
By the previous part, $\Theta$ is a representation of $\SS^{at}$  with respect to the signature 
$\tau\cap\set{\compo, \D, \R, \Fix,0,1'}$  over a set of size at most 
$(|S^{at}|-1)^{|S^{at}|}\leq (|S|-1)^{|S|}$.
Now,  as in the proof of Theorem~\ref{thm:main}, we may extend $\Theta$ to a representation 
$\phi$ of $\SS$ over the same base by letting 
\[
s^\phi=\bigcup \{(d\compo s)^\Theta\mid d\in At(D(S))\}
\]
see \eqref{eq:phisharp}.
Similarly to the proof of Theorem~\ref{thm:main}, $\phi$ respects antidomain,
thus $\SS$ is representable over a set of size at most $(|S|-1)^{|S|}$.

All of this ignored $\sqcup$ and ${}^\uparrow$ if they were present in $\tau$ (with $\A\in \tau$).  
For these operations observe that if $\sqcup\in \tau$ then as $\SS\in \F(\tau)$, 
then laws \eqref{eq:sqcupdom} and \eqref{eq:uniondef} will hold, 
so that the ``moreover'' statement of Theorem~\ref{thm:cupmi} part~(1) ensures the correct 
representation of~$\sqcup$.  If both~$\sqcup$ and~${}^\uparrow$ are in $\tau$, 
then laws \eqref{eq:sqcupdom}--\eqref{eq:dynamic} hold and the ``moreover'' statement 
of Theorem~\ref{thm:cupmi} part~(2) shows that both~$\sqcup$  and~${}^\uparrow$ are correctly 
represented provided that~$\compo$ and $\A$ have been correctly represented 
(which we showed could be achieved on a set of size at most $(|S|-1)^{|S|}$). 
\end{proof}

Of course if $S$ is infinite, an application of the downward L\"owenheim--Skolem Theorem 
yields a representation on a base set of size at most $|S|$.

In general, intersection is not preserved by the representation method in Theorem~\ref{thm:dr}.
If $\bar a$ is in the domain of $x^\theta$ and $y^\theta$ with $x\cdot y=0$
then it is still possible that the view of $x^\theta(\bar a)$ coincides with 
the view of $y^\theta(\bar a)$.  
This occurs in Example~\ref{eg:ab} when $x=a$, $y=b$ and $\bar a=(a,b,d)$ for example.

Most of the cases in the following theorem can be extracted from~\cite{jacsto:modal}, 
but the proof here gives an alternative perspective.

\begin{theorem}
Let  $\compo\in\tau\subseteq\set{0, \compo, \cdot, \sqcup, \D, \A, \Fix, {}^\uparrow, 1'}$ 
\up(but no range operation\up).  
If $\SS\in\F(\tau)$ is finite then it has a representation on a base of size at most $|S|^3$.
\end{theorem}
\begin{proof} 
Let $\SS\in\F(\tau)$ be a finite algebra of functions over a base $X$.  
For each pair $(s, t)$ where  $s\neq t\in \SS$, let $x_{s, t}$ be an arbitrary element of $X$ 
such that $s$ disagrees with $t$ at $x_{s, t}$, so either both $s, t$ are defined at 
$x_{s, t}$ and $s(x_{s, t})\neq t(x_{s, t})$ or just one of them is defined 
(such a point must exist, since $s, t$ are distinct functions).  
Let
\[
Y=\set{x_{s, t}, u(x_{s, t})\mid s, t, u\in \SS,\; s\neq t,\; u\mbox{ defined at } x_{s, t}}.
\]
Recall from Theorem~\ref{thm:eq th} that $\SS\restr Y$ is the algebra of functions 
$\set{s\cap (Y\times Y)\mid s\in\SS}$ with operations obtained from $\SS$ by relativization.    
Clearly $|Y|\leq |\SS|^3$.  We must show that the map $\theta:s\mapsto s\cap(Y\times Y)$ is 
an isomorphism from $\SS$ to $\SS\restr Y$.

Observe
\begin{equation}\label{eq:restr2}
y\in Y,\; v\in\SS, \; v(y) \mbox{ defined } \Rightarrow v(y)\in Y
\end{equation}
because if $y=u(x_{s, t})$ then $v(y)=v(u(x_{s, t}))=(u\compo v)(x_{s, t})\in Y$.

For any $s\neq t\in \SS$, $s\cap(Y\times Y)$ disagrees with $t\cap(Y\times Y)$ at $x_{s, t}$, 
so $\theta(s)\neq \theta(t)$ and $\theta$ is injective.  
Clearly $0\cap(Y\times Y)=0,\; 1'\cap(Y\times Y)=1'_Y$ and 
$(s\cap(Y\times Y))\cdot(t\cap(Y\times Y))=(s\cdot t)\cap (Y\times Y)$, 
so $\theta$ respects $0, 1', \cdot$.  
Equation~\eqref{eq:restr2} shows that $\theta$ also respects 
$\compo, \sqcup, \D, \A, \Fix, {}^\uparrow$.  Take composition, for example.
\begin{align*}
(x, y)\in \theta(u;v)&\iff x, y\in Y \wedge (x, y)\in (u;v)\\
&\iff x\in Y\wedge y=(u;v)(x)=v(u(x))\;\;&\mbox{ (by \eqref{eq:restr2})}\\
&\iff (x, u(x))\in \theta(u)\wedge (u(x), y)\in \theta(v) \;\;&\mbox{ (by \eqref{eq:restr2})}\\
&\iff (x, y)\in \theta(u) \compo \theta(v)
\end{align*}
Similarly, for maximum iterate, we have the following.
\begin{align*}
(x, y)\in\theta( u^\uparrow) &\iff x\in Y \wedge (x, y)\in  u^\uparrow\\
&\iff x\in Y \wedge (\exists k\geq 0)\exists x_0, \ldots, x_k\bigwedge_{i<k}(x_i, x_{i+1})\in u \wedge x=x_0\wedge y=x_k\not\in \dom(u)\\
&\iff \exists x_0,\ldots, x_k\bigwedge_{i<k}(x_i, x_{i+1})\in\theta(u)\wedge x=x_0\wedge y=x_k\not\in\dom(\theta(u))\\
&\iff (x, y)\in(\theta(u))^\uparrow
\end{align*}
Checking the preservation of the other operations is similar.
Hence $\theta$ is an isomorphism.
\end{proof}

Similarly, the finite representation property is easy to establish for signatures 
that cannot express $\D$.
This leaves one group of cases.

\begin{problem} 
Let 
$\set{\compo,\meet,\D, \R}\subseteq\tau\subseteq
\set{\compo,\meet, \D, \A, \R, \Fix,\sqcup, {}^\uparrow,0,1'}$.  
Is it the case that every finite member of $\F(\tau)$ has a representation on a finite base?
In particular, does the finite representation property hold for the signature 
$\tau=\set{\compo,\meet, \A, \R,0}$ and the signature $\set{\compo,\meet, \D, \R,0}$?
\end{problem}

\end{document}